\documentclass[11pt]{article}

\usepackage[margin=1in]{geometry}
\usepackage[utf8]{inputenc}
\usepackage[english]{babel}
\usepackage{amsmath,amssymb,amsthm,mathtools}
\usepackage[hidelinks]{hyperref}
\usepackage[nameinlink,capitalize]{cleveref}
\numberwithin{equation}{section}

\newtheorem{theorem}{Theorem}[section]
\newtheorem{proposition}[theorem]{Proposition}
\newtheorem{lemma}[theorem]{Lemma}
\newtheorem{corollary}[theorem]{Corollary}
\theoremstyle{definition}
\newtheorem{definition}[theorem]{Definition}
\theoremstyle{remark}
\newtheorem{remark}[theorem]{Remark}

\newcommand{\R}{\mathbb{R}}
\newcommand{\Sph}{\mathbb{S}}
\newcommand{\Pcal}{\mathbb{P}}
\newcommand{\diver}{\operatorname{div}}
\newcommand{\curl}{\operatorname{curl}}
\newcommand{\supp}{\operatorname{supp}}
\newcommand{\weakstar}{\xrightharpoonup{~\ast~}}
\newcommand{\1}{\mathbf{1}}
\newcommand{\abs}[1]{\left|#1\right|}
\newcommand{\norm}[1]{\left\|#1\right\|}
\newcommand{\pair}[2]{\left\langle #1,#2\right\rangle}

\title{Self-Similar Solutions of the Two-Dimensional\\
Incompressible Euler Equation from Large Initial Data}
\author{Hyungjun Choi\footnote{Department of Mathematics, Princeton University, Princeton, USA, hyungjun.choi@princeton.edu}}
\date{\today}

\begin{document}
\maketitle

\begin{abstract}
Given $0<a<1$ and any $C^1$, divergence-free, $(-a)$-homogeneous vector field $u_0$ on the punctured plane, we construct a forward self-similar solution of the two-dimensional incompressible Euler equations on the plane with initial datum $u_0$. No smallness or sign condition is required. The solution is obtained as a vanishing-dissipation limit of hypodissipative self-similar profiles. The key estimate is a uniform vorticity profile bound in the critical Lorentz space $L^{\frac{2}{1+a},\infty}$. The resulting velocity has locally finite energy and converges strongly to $u_0$ locally in $L^2$ as time approaches zero. The vorticity preserves its distribution function and converges strongly to the initial vorticity locally in $L^{\frac{2}{1+a}-}$ as time approaches zero.
\end{abstract}

\section{Introduction}

We study the initial-value problem for the two-dimensional incompressible Euler equations
\begin{equation} \label{eq:euler}
\partial_t u+u\cdot\nabla u+\nabla p=0,\qquad \diver u=0\quad\text{in }\R^2\times(0,\infty).
\end{equation}
Upon introducing the scalar vorticity $\omega=\curl u=\partial_1u_2-\partial_2u_1$, taking the curl of \eqref{eq:euler} gives the transport equation below. In the decaying Lorentz class used in this paper, the velocity is recovered with the indicated Biot--Savart normalization:
\begin{equation} \label{eq:vorticity-introduction}
\partial_t\omega+u\cdot\nabla\omega=0,\qquad u=K*\omega,\qquad K(x)=\frac{1}{2\pi}\frac{x^\perp}{|x|^2},\qquad x^\perp=(-x_2,x_1).
\end{equation}
The absence of vortex stretching is the fundamental distinction between the two- and three-dimensional equations. Global-in-time existence of classical two-dimensional solutions was established in the foundational works of Wolibner and H\"older \cite{Wolibner1933,Holder1933}. In modern continuation theory, the Beale--Kato--Majda criterion \cite{BealeKatoMajda1984} states that a smooth Euler solution can lose regularity only if
\[ \int_0^T\norm{\omega(t)}_{L^\infty}\,dt=\infty. \]
Since \eqref{eq:vorticity-introduction} transports vorticity, $\norm{\omega(t)}_{L^\infty}$ is conserved for smooth two-dimensional solutions, and the local classical solution therefore extends globally. We refer to \cite{MajdaBertozzi2002} for a systematic account of the classical theory.

The low-regularity existence and uniqueness theory is due to Yudovich. For initial vorticity $\omega_0\in L^1 \cap L^\infty(\R^2)$, he constructed a unique global weak solution \cite{Yudovich1963}. Yudovich later enlarged the uniqueness class to certain mildly unbounded vorticities \cite{Yudovich1995}; related borderline Besov-space results were obtained by Vishik \cite{Vishik1999}. Thus, bounded vorticity is not an absolute endpoint, and uniqueness is conjectured to persist for broader classes of unbounded vorticity.

Weak-solution existence extends substantially below the Yudovich class. DiPerna and Majda developed concentration-cancellation methods for regularized two-dimensional Euler flows \cite{DiPernaMajda1987,DiPernaMajda1988}. Their framework yields global weak solutions under broad compactness hypotheses, including the regime $\omega_0\in L^1\cap L^p(\R^2)$ with $p>1$. At the measure level, Delort proved global existence for vorticity that is a bounded Radon measure of distinguished sign and has locally finite energy \cite{Delort1991}; see also Schochet's weak-vorticity formulation and concentration-cancellation treatment \cite{Schochet1995}. Consequently, global weak existence is available for many rough data classes, whereas uniqueness outside the Yudovich regime is one of the central unresolved themes of two-dimensional ideal flow.

Without such vorticity structure, distributional solutions have been shown to be nonunique: compactly supported wild solutions were constructed by Scheffer \cite{Scheffer1993} and Shnirelman \cite{Shnirelman1997}. Convex integration, introduced by De Lellis and Sz\'ekelyhidi, subsequently placed this flexibility in a systematic framework \cite{DeLellisSzekelyhidi2009}. Applications of convex integration to the two-dimensional Euler equations can be found in \cite{GiriRadu2024} and \cite{BrueColomboKumar2026}; in particular, the latter construction, due to Bru\`{e}, Colombo, and Kumar, provides vorticity bounded in $L^\infty_tL^p_x$ for some $p>1$.

The present paper concerns a scale-invariant part of the theory. The Euler equations admit a one-parameter family of space-time scalings. For every fixed $a\in(0,1)$, they are invariant under
\[ u(t,x)\longmapsto u_\lambda(t,x)\triangleq \lambda^a u(\lambda x,\lambda^{1+a}t), \qquad p(t,x)\longmapsto p_\lambda(t,x)\triangleq \lambda^{2a}p(\lambda x,\lambda^{1+a}t). \]
We call a solution self-similar when it is unchanged by this transformation for every $\lambda>0$. Such a flow is completely determined by the single profile $(U,P)=(u,p)|_{t=1}$ and necessarily has the form
\[ u(t,x)=t^{-\frac{a}{1+a}} U\left(\frac{x}{t^{{\frac{1}{1+a}}}}\right), \qquad p(t,x)=t^{-\frac{2a}{1+a}} P\left(\frac{x}{t^{{\frac{1}{1+a}}}}\right). \]
As $t\downarrow0$, a fixed spatial point $x\neq0$ probes the far field of the profile, since $x/t^{1/(1+a)}\to\infty$. The initial condition is therefore encoded as a boundary condition at spatial infinity for $(U,P)$. Scale invariance requires the initial velocity to be $(-a)$-homogeneous:
\[ u_0(x)=|x|^{-a}\overline u_0\left(\frac{x}{|x|}\right), \qquad \diver u_0=0. \]
Thus the self-similar problem asks whether a prescribed angular profile $\overline u_0$ can be connected to a globally defined profile $U$ that solves the stationary equation in similarity variables.

Self-similar solutions have long served as a meeting point between existence theory, singularity formation, and nonuniqueness.

For the three-dimensional Navier--Stokes equations, the analogous spatial homogeneity is $a=1$. Jia and \v{S}ver\'ak constructed forward self-similar solutions for arbitrarily large $(-1)$-homogeneous data \cite{JiaSverak2014}. The corresponding large-data problem for the two-dimensional Navier--Stokes equations is rather different. Although finite-energy solutions are much better behaved in two dimensions, a nonzero $(-1)$-homogeneous velocity has infinite local energy near the origin. Albritton, Guillod, Korobkov, and Ren recently constructed forward self-similar solutions from arbitrary smooth $(-1)$-homogeneous data \cite{AlbrittonGuillodKorobkovRen2026}; an independent construction was given by Gui, Liu, and Xie \cite{GuiLiuXie2026}. Hou and Song subsequently treated the hypodissipative equations with fractional diffusion \cite{HouSong2026}.

For the two-dimensional Euler equations, self-similar solutions may be singular and can display spiral geometry. Elling constructed algebraic spiral solutions with high-frequency rotational symmetry \cite{Elling2013,Elling2016}. Shao, Wei, and Zhang further developed Elling's method for low-frequency symmetry \cite{ShaoWeiZhang2025}, while Choi constructed such spirals without imposing rotational symmetry \cite{Choi2025}. See also \cite{Gomez-SerranoGarcia2026} for the construction of algebraic spiral solutions for the generalized SQG equation. Moreover, Choi and Coiculescu recently constructed and classified homogeneous steady states, which are automatically self-similar \cite{ChoiCoiculescu2026}. With the exception of such special homogeneous configurations, the known existence theories for spiral solutions are organized around radial power-law vortices: the allowed angular dependence is perturbative or symmetry-dominated.

Forward self-similar solutions are likewise closely tied to potential nonuniqueness. Jia and \v{S}ver\'ak proposed that a bifurcation of self-similar Navier--Stokes solutions could produce multiple Leray--Hopf weak solutions with the same initial datum \cite{JiaSverak2015}. Guillod and \v{S}ver\'ak provided numerical evidence for such a bifurcation \cite{GuillodSverak2023}. More recently, Hou, Wang, and Yang announced a computer-assisted proof \cite{HouWangYang2025}, while Ionescu, Jia, and Palasek proposed another scenario with additional symmetry \cite{IonescuJiaPalasek2025}. The aforementioned works on two-dimensional self-similar Navier--Stokes equations also provide numerical evidence of nonuniqueness. We also note that Coiculescu and Palasek proved nonuniqueness at critical regularity using a novel multiscale construction \cite{CoiculescuPalasek2026}. For the two-dimensional Euler equations, Bressan and Shen reported numerically that distinct self-similar profiles can arise from different truncations of the same initial datum \cite{BressanShen2021}. The existence of a large self-similar profile supplies a base state on which instability and bifurcation scenarios can be investigated.

On the other hand, backward self-similar solutions are natural candidates for finite-time blowup, which is another reason self-similarity is central in fluid mechanics. For the three-dimensional Navier--Stokes equations, nontrivial exactly self-similar blowup solutions satisfying natural integrability or local-energy conditions were excluded by Ne\v{c}as, R\r{u}\v{z}i\v{c}ka, and \v{S}ver\'ak \cite{NecasRuzickaSverak1996}, and by Tsai \cite{Tsai1998}. For the three-dimensional Euler equations, Elgindi proved finite-time self-similar singularity formation for $C^{1,\alpha}$ velocities \cite{Elgindi2021}. A related $C^{1,\alpha}$ blowup result with boundary was obtained by Chen and Hou \cite{ChenHou2023}. Motivated by the numerical work of Luo and Hou \cite{LuoHou2014}, Chen and Hou gave a computer-assisted proof of finite-time blowup for smooth data with boundary \cite{ChenHouAnalysis,ChenHouNumerics}. For recent advances in the numerical discovery of profiles, see \cite{BGG-SJL2025} and the references therein. In addition, Constantin, Ignatova, and Vicol ruled out a range of scaling parameters for a potential finite-energy self-similar blowup \cite{ConstantinIgnatovaVicol2026}.

Our main theorem provides self-similar profiles for the two-dimensional Euler equations for every $C^1$, $(-a)$-homogeneous divergence-free datum throughout the range $0<a<1$, with no smallness or sign assumption. The emphasis is threefold. First, in contrast with the previously known spiral constructions, the angular data need not be close to a radial vortex or satisfy a radial dominance condition. Second, the resulting flow is a global-in-time weak solution of the unforced two-dimensional Euler equations without finite total-vorticity control. The critical control is instead $\omega\in L^\infty_tL^{\frac{2}{1+a},\infty}_x$. Third, the construction supplies genuinely large self-similar Euler profiles and thus gives mathematical support to the nonuniqueness program. It does not by itself prove nonuniqueness, but it furnishes the class of large base solutions needed for a future instability or bifurcation analysis. Several months after this manuscript was announced, Hou, Qin, Song, and Zhang proved a similar theorem using a Lagrangian flow method \cite{HouQinSongZhang2026}.

As a byproduct, we prove the existence of forward self-similar solutions for the two-dimensional $(-\Delta)^{\frac{1+a}{2}}$-hypodissipative Navier--Stokes equations for every $C^1$, $(-a)$-homogeneous divergence-free datum when $0<a<1$. Our proof strategy is completely different from that of Hou and Song \cite{HouSong2026} and establishes smoothness of the profile throughout the full parameter range. See \emph{Remark \ref{rem:hypodiss-comparison}} for further discussion in this direction.

\subsection{Statement of the result}

\begin{definition}\label{def:weak}
We use the following notions of solutions on $\R^2\times[0,\infty)$.
\begin{enumerate}
\item[(a)] Let $u_0\in L^2_{\rm loc}(\R^2)$ be divergence-free. A vector field $u\in L^2_{\rm loc}(\R^2\times[0,\infty))$ is a weak solution of the two-dimensional Euler equations with initial datum $u_0$ if $\diver u=0$ and, for every divergence-free $\varphi\in C_c^\infty(\R^2\times[0,\infty);\R^2)$,
\[ \int_0^\infty\int_{\R^2}\left(u\cdot\partial_t\varphi+(u\otimes u):\nabla\varphi\right)\,dx\,dt+\int_{\R^2}u_0(x)\cdot\varphi(x,0)\,dx=0. \]
\item[(b)] Let $u\in L^1_{\rm loc}(\R^2\times[0,\infty))$ be divergence-free and let $\omega_0\in L^1_{\rm loc}(\R^2)$. A scalar $\omega\in L^1_{\rm loc}(\R^2\times[0,\infty))$ is a weak solution of the vorticity equation with velocity $u$ and initial datum $\omega_0$ if $u\omega\in L^1_{\rm loc}(\R^2\times[0,\infty))$ and, for every $\psi\in C_c^\infty(\R^2\times[0,\infty))$,
\[ \int_0^\infty\int_{\R^2}\omega\bigl(\partial_t\psi+u\cdot\nabla\psi\bigr)\,dx\,dt+\int_{\R^2}\omega_0(x)\psi(x,0)\,dx=0. \]
\item[(c)] With $u$ and $\omega_0$ as in (b), a measurable scalar $\omega$ is a renormalized solution of the vorticity equation if, for every $\beta\in C^1(\R)\cap L^\infty(\R)$, the function $\beta(\omega)$ is a weak solution of $\partial_t \beta(\omega) + u \cdot \nabla \beta(\omega) = 0$ with initial datum $\beta(\omega_0)$.
\end{enumerate}
\end{definition}

\begin{theorem}\label{thm:main}
Let $0<a<1$, and suppose that $\overline u_0\in C^1(\Sph^1;\R^2)$ is such that
\[ u_0(x)=|x|^{-a}\overline u_0\left(\frac{x}{|x|}\right), \qquad \diver u_0=0. \]
Write
\[ \omega_0\triangleq \curl u_0 =|x|^{-1-a} \overline \omega_0\left(\frac{x}{|x|}\right),\qquad p_c\triangleq \frac{2}{1+a},\qquad q_c\triangleq \frac2a. \]
Then there exist a divergence-free profile $U$ and a pressure $P$ with the following properties.
\begin{enumerate}
\item[(a)] The pair $(U,P)$ solves
\begin{equation} \label{eq:euler-profile}
-\frac{1}{1+a}\bigl(aU+y\cdot\nabla U\bigr)+U\cdot\nabla U+\nabla P=0,\qquad \diver U=0
\end{equation}
in distributions on $\R^2$. Moreover,
\[ U\in L^{q_c,\infty}(\R^2),\qquad \nabla U \in L^{p_c,\infty}(\R^2),\qquad P\in L^{{\frac{1}{a}},\infty}(\R^2). \]
\item[(b)] Define the self-similar velocity, pressure, and vorticity by
\begin{equation} \label{eq:ss-ansatz}
u(t,x)=t^{-\frac{a}{1+a}} U\left(\frac{x}{t^{{\frac{1}{1+a}}}}\right),\quad p(t,x)=t^{-\frac{2a}{1+a}} P\left(\frac{x}{t^{{\frac{1}{1+a}}}}\right), \quad \omega(t,x)=t^{-1} \Omega\left(\frac{x}{t^{{\frac{1}{1+a}}}}\right)
\end{equation}
for $t>0$, where $\Omega=\curl U$. Then $u$ is a weak solution of the Euler equations with initial datum $u_0$, and $\omega$ is a renormalized solution of the vorticity equation with velocity $u$ and initial datum $\omega_0$, in the senses of \cref{def:weak}(a) and (c), respectively. Moreover,
\[ u\in C\bigl([0,\infty);L^2_{\rm loc}(\R^2)\bigr),\qquad u(0)=u_0, \]
and, for every $1\leq r<p_c$,
\[ \omega\in C\bigl([0,\infty);L^r_{\rm loc}(\R^2)\bigr),\qquad \omega(0)=\omega_0. \]
The critical Lorentz quasi-norm is conserved:
\[ \norm{\omega(t)}_{L^{p_c,\infty}}=\norm{\Omega}_{L^{p_c,\infty}}=\norm{\omega_0}_{L^{p_c,\infty}},\qquad t>0, \]
and $\omega(t)\weakstar\omega_0$ in $L^{p_c,\infty}(\R^2)$ as $t\downarrow0$.
\end{enumerate}
Here the weak-$L^r$ quasi-norm is defined by
\[ \norm{f}_{L^{r,\infty}} \triangleq \sup_{\lambda>0} \lambda\,\abs{\{x:\abs{f(x)}>\lambda\}}^{\frac{1}{r}}, \quad r \in [1,\infty). \]
\end{theorem}
\noindent The construction applies to every $C^1$ divergence-free angular datum $\overline u_0$ and requires neither a smallness condition nor a sign condition.

\begin{remark}[The endpoint $a=0$]\label{rem:a-zero}
At $a=0$, one has $q_c=\infty$, and neither the Riesz transforms nor the Biot--Savart normalization are controlled in the required $L^\infty$ endpoint class. In addition, a constant harmonic velocity is not excluded by $L^\infty$. Thus the strict inequality $a>0$ is essential to this construction.
\end{remark}

\begin{remark}[Vorticity equation] \label{rem:transport-range}
For $\frac{1}{2}\leq a<1$, the velocity formulation of Euler in \cref{thm:main} remains valid. However,
\[ \frac1{p_c}+\frac1{q_c}\geq1, \]
so the estimates do not imply $U\Omega\in L^1_{\rm loc}$. On the other hand, in the range $0<a<\frac{1}{2}$, the vorticity equation holds in the sense of distributions.
\end{remark}

\begin{remark}[Renormalization]\label{rem:renormalization}
The velocity $u$ satisfies the Sobolev regularity and growth assumptions of the DiPerna--Lions theory \cite{DiPernaLions1989}. Together with renormalization, this identifies $\omega$ with the solution transported by the measure-preserving regular Lagrangian flow $X$ of $u$: for every $t\geq0$, $\omega(t,X(t,x))=\omega_0(x)$ for almost every $x$. In particular, the distribution function of $\omega$ is conserved. We establish renormalization and the strong local vorticity trace in \cref{prop:renormalization} by adapting the duality argument of Crippa and Spirito \cite{CrippaSpirito2015} to fractional dissipation and $L^{p_c,\infty}$ vorticity.
\end{remark}

\begin{corollary} \label{cor:vorticity}
Let $0<a<\frac{1}{2}$, and let $U,\Omega$ and $u,\omega$ be the profiles and self-similar fields, respectively, supplied by \cref{thm:main}. Then
\[ U\Omega\in L^1_{\rm loc}(\R^2), \]
and the profile vorticity satisfies
\begin{equation} \label{eq:euler-vort-profile}
\left(U-\frac{y}{1+a}\right)\cdot\nabla\Omega=\Omega
\end{equation}
in distributions. The physical vorticity $\omega$ is a weak solution of $\partial_t \omega + u \cdot \nabla \omega = 0$ with initial datum $\omega_0$ in the sense of \cref{def:weak}(b).
\end{corollary}

\subsection{Strategy of proof}

Substitution of the self-similar ansatz \eqref{eq:ss-ansatz} into \eqref{eq:euler} gives \eqref{eq:euler-profile}. Writing $\Lambda=(-\Delta)^{1/2}$, we obtain its solution as the inviscid limit of smooth profiles satisfying
\begin{equation} \label{eq:eps-profile}
\epsilon\Lambda^{1+a}U_\epsilon-\frac{1}{1+a}\bigl(aU_\epsilon+y\cdot\nabla U_\epsilon\bigr)+U_\epsilon\cdot\nabla U_\epsilon+\nabla P_\epsilon=0,\qquad \diver U_\epsilon=0.
\end{equation}
In Section 2, we construct a smooth hypodissipative profile from $(-a)$-homogeneous initial data and establish an estimate for $\norm{\Omega_\epsilon}_{L^{p_c,\infty}}$ that is uniform in $\epsilon$. In Sections 3 and 4, we prove the existence of an inviscid limit and show that the limit is a self-similar Euler solution with the same initial data. Section~\ref{sec:trace-renormalization} also establishes renormalization and the strong local vorticity trace.

\section{The hypodissipative profiles}

\begin{theorem}[Hypodissipative profile]\label{thm:critical-profile}
Let $0<a<1$ and let $v_0$ be a $C^1_{\rm loc}$, divergence-free, $(-a)$-homogeneous vector field on $\R^2\setminus\{0\}$. Put $\omega_0=\curl v_0$. There are smooth profiles $W,Q$ such that $\diver W=0$, $W=K*\Omega$, $\Omega=\curl W$, and
\begin{equation} \label{eq:unit-dissipative-profile}
\Lambda^{1+a}W-\frac1{1+a}\bigl(aW+y\cdot\nabla W\bigr)+W\cdot\nabla W+\nabla Q=0.
\end{equation}
Moreover,
\begin{equation} \label{eq:unit-critical-bound}
\norm{\Omega}_{L^{p_c,\infty}}\leq C_a\norm{\omega_0}_{L^{p_c,\infty}},\qquad \Omega\in L^\infty(\R^2).
\end{equation}
The associated self-similar fields
\[ v(t,x)=t^{- \frac{a}{1+a}} W(t^{-\frac{1}{1+a}} x),\qquad \omega(t,x)=t^{-1}\Omega(t^{-\frac{1}{1+a}}x) \]
solve the hypodissipative equations for $t>0$ and satisfy
\[ \omega(t)\weakstar\omega_0\text{ in }L^{p_c,\infty},\qquad v(t)\weakstar v_0\text{ in }L^{q_c,\infty}\quad\text{as }t\downarrow0. \]
\end{theorem}

\begin{remark}[Comparison between Theorem~\ref{thm:critical-profile} and Hou--Song \cite{HouSong2026}]\label{rem:hypodiss-comparison}
Theorem~\ref{thm:critical-profile} produces a smooth hypodissipative profile throughout the full range $0<a<1$, whereas the smoothness conclusion in the corresponding result of Hou--Song \cite{HouSong2026} is stated for $\frac{1}{3}<a<1$. The present proof works directly with the vorticity in the critical Lorentz space and, crucially for the inviscid limit, gives an estimate that remains uniform after the dissipation parameter is rescaled. Appendix~\ref{app:galerkin} records a complementary velocity-based construction by a Galerkin scheme that gives a streamlined proof of their profile-existence result.
\end{remark}

\begin{lemma}[Auxiliary Biot--Savart estimates]\label{lem:biot-aux}
If $f\in L^{p_c,\infty}(\R^2)\cap L^\infty(\R^2)$, then
\[ \norm{K \ast f}_{L^{q_c,\infty}}\leq C_a\norm{f}_{L^{p_c,\infty}},\qquad \norm{K \ast f}_{L^\infty}\leq C_a\norm{f}_{L^{p_c,\infty}}^{\frac1{1+a}}\norm{f}_{L^\infty}^{\frac{a}{1+a}}. \]
Moreover, for every $r>p_c$,
\[ \norm{f}_{L^r}\leq C_{a,r}\norm{f}_{L^{p_c,\infty}}^{\frac{p_c}{r}}\norm{f}_{L^\infty}^{1-\frac{p_c}{r}}. \]
Consequently, if $0<\eta<1$ and $r>2/(1-\eta)$, then $K \ast f\in C_b^\eta$ and
\[ [K \ast f]_{\dot{C}^\eta}\leq C_{a,\eta,r}\left(\norm{f}_{L^{p_c,\infty}}^{\frac1{1+a}}\norm{f}_{L^\infty}^{\frac{a}{1+a}}+\norm{f}_{L^{p_c,\infty}}^{\frac{p_c}{r}}\norm{f}_{L^\infty}^{1-\frac{p_c}{r}}\right). \]
\end{lemma}

\begin{proof}
The first estimate is the Lorentz Hardy--Littlewood--Sobolev estimate in \cref{prop:biot-lorentz}. To prove the $L^\infty$ estimate, split the convolution at radius $R$. The inner integral is bounded by $CR\norm{f}_{L^\infty}$, while the outer integral is bounded by
\[ \int_{\abs z\geq R}\frac{\abs{f(x-z)}}{\abs z}\,dz\leq C\norm{f}_{L^{p_c,\infty}}\norm{\abs z^{-1}\1_{\{\abs z\geq R\}}}_{L^{p_c',1}}\leq C\norm{f}_{L^{p_c,\infty}}R^{-a}. \]
Optimizing in $R$ proves the second estimate. Finally, $\nabla (K \ast f)$ is a matrix of Riesz transforms, so its $L^r$ norm is bounded by $C_r\norm f_{L^r}$ for any $r\in (1,\infty)$. Morrey's inequality gives the stated $\dot{C}^\eta$ estimate.
\end{proof}

The following lemma is due to Constantin and Vicol \cite{ConstantinVicol2012}.

\begin{lemma}[Nonlinear maximum principle] \label{lem:fractional-maximum}
Let $0<\sigma<2$, $1<p<\infty$, and let $f\in C^\infty(\R^2)\cap L^{p,\infty}(\R^2)$ tend to zero at infinity. If $M=f(x_*)=\max f>0$, then
\begin{equation} \label{eq:fractional-maximum}
\Lambda^\sigma f(x_*)\geq c_{\sigma,p}\frac{M^{1+\sigma p/2}}{\norm{f}_{L^{p,\infty}}^{\sigma p/2}}.
\end{equation}
The analogous estimate holds for $-f$ at a negative minimum.
\end{lemma}

\begin{proof}
Put $G=\norm f_{L^{p,\infty}}$ and $E=\{z \in \R^2 : f(x_\ast + z)>\frac{M}{2}\}$. Then $\abs E\leq(2G/M)^p$. Choose $R$ so that $\abs{B_R}=2\abs E$. At the maximum,
\[ \Lambda^\sigma f(x_*)=c_\sigma\,\mathrm{p.v.}\!\int_{\R^2}\frac{M-f(x_*+z)}{\abs z^{2+\sigma}}\,dz\geq c_\sigma\int_{B_R\setminus E}\frac{M/2}{\abs z^{2+\sigma}}\,dz. \]
Among subsets of $B_R$ of measure at least $\abs{B_R}/2$, the integral of the radially decreasing function $\abs z^{-2-\sigma}$ is smallest on an outer annulus. Hence the last integral is at least $c_{\sigma,p}MR^{-\sigma}$. Since $R^2\leq C_p(G/M)^p$, this implies \eqref{eq:fractional-maximum}.
\end{proof}

Fix $0<\eta<1$ and $T>0$. Define the drift space
\[ B_T\triangleq\left\{b\in C((0,T];C_b^\eta(\R^2;\R^2)):\ \diver b=0,\ \nabla b(t)\in L^{p_c,\infty}(\R^2)\text{ for }0<t\leq T,\ \norm{b}_{B_T}<\infty\right\}, \]
where
\begin{equation} \label{eq:BT-norm}
\norm{b}_{B_T}\triangleq\max\left\{\sup_{0<t\leq T}t^{\frac{a}{1+a}}\norm{b(t)}_{L^\infty},\ \sup_{0<t\leq T}t^{\frac{a+\eta}{1+a}}[b(t)]_{\dot C^\eta},\ \sup_{0<t\leq T}\norm{\nabla b(t)}_{L^{p_c,\infty}}\right\}.
\end{equation}
Thus membership in $B_T$ permits the scale-invariant growth rates $\norm{b(t)}_{L^\infty}=O(t^{-a/(1+a)})$ and $[b(t)]_{\dot C^\eta}=O(t^{-(a+\eta)/(1+a)})$ as $t\downarrow0$. Both rates are integrable in time, and the last term in \eqref{eq:BT-norm} implies $b\in L^1((0,T);W^{1,s}_{\rm loc})$ for every $1<s<p_c$.

\begin{lemma}[Linear drift--diffusion]\label{lem:linear-drift}
Let $b\in B_T$ and $f\in L^{p_c,\infty}(\R^2)$. Then the following assertions hold.
\begin{enumerate}
\item[(a)] There is a distributional solution of
\begin{equation} \label{eq:linear-drift}
\partial_t\vartheta+b\cdot\nabla\vartheta+\Lambda^{1+a}\vartheta=0,\qquad \vartheta(0)=f,
\end{equation}
in the class
\[ \sup_{0<t\leq T}\norm{\vartheta(t)}_{L^{p_c,\infty}}<\infty,\qquad \sup_{0<t\leq T}t\norm{\vartheta(t)}_{L^\infty}<\infty, \]
where $\vartheta(t)\weakstar f$ in $L^{p_c,\infty}$ as $t\downarrow0$. Moreover, it satisfies the bounds
\begin{equation} \label{eq:linear-critical-bounds}
\norm{\vartheta(t)}_{L^{p_c,\infty}}\leq C_a\norm f_{L^{p_c,\infty}},\qquad \norm{\vartheta(t)}_{L^\infty}\leq C_at^{-1}\norm f_{L^{p_c,\infty}}.
\end{equation}
\item[(b)] Every distributional solution of \eqref{eq:linear-drift} satisfying \eqref{eq:linear-critical-bounds} and $\vartheta(t)\weakstar f$ in $L^{p_c,\infty}$ as $t\downarrow0$ obeys
\begin{equation} \label{eq:linear-trace-rate}
\abs{\pair{\vartheta(t)-f}{\phi}}\leq C_{a,\norm{b}_{B_T},\phi}\norm{f}_{L^{p_c,\infty}}(t + t^{\frac{1}{1+a}}),\qquad \phi\in C_c^\infty(\R^2),\quad 0<t\leq T.
\end{equation}
\item[(c)] Put $S(t)=e^{-t\Lambda^{1+a}}$. For every $0<s\leq t\leq T$, the following identity holds in $\mathcal S'(\R^2)$:
\begin{equation} \label{eq:linear-variation-constants}
\vartheta(t)=S(t-s)\vartheta(s)-\int_s^t\nabla S(t-r)\cdot\bigl(b(r)\vartheta(r)\bigr)\,dr.
\end{equation}
Moreover, for every $0<\tau<T$ and $0<\beta<a$, the solution is bounded in $C^\beta(\R^2)$ on $[\tau,T]$.
\item[(d)] Suppose that $b_n,b\in B_T$, that $b_n$ is uniformly bounded in $B_T$, and that $b_n\to b$ locally uniformly on $(0,T]\times\R^2$. Let $\vartheta_n$ be solutions with the same initial datum $f$ and the common bounds \eqref{eq:linear-critical-bounds}. Then a subsequence converges locally uniformly on $(0,T]\times\R^2$, and weak-star in both $L^{p_c,\infty}$ and $L^\infty$ at each positive time, to a solution with drift $b$, the same initial datum, and the same bounds.
\item[(e)] If $b\in B_T$ for every $T>0$ and $\displaystyle\sup_{T>0}\norm{b}_{B_T}<\infty$, then there is a solution of \eqref{eq:linear-drift} on all of $(0,\infty)$ satisfying the global critical bounds in \eqref{eq:linear-critical-bounds}.
\end{enumerate}
\end{lemma}

\noindent The locally uniform convergence in $(0,T]\times\R^2$ means uniform convergence on $[\tau, T] \times K$ for any $\tau > 0$ and compact $K \Subset \R^2$.

\begin{proof}
\noindent\emph{Approximation and uniform bounds.}\par
If $f=0$, take $\vartheta\equiv0$. Hence assume $G=\norm f_{L^{p_c,\infty}}>0$. We first construct the smooth drifts used in the approximation. Choose $\chi\in C^\infty(\R;[0,1])$ with $\chi(s)=0$ for $s\leq1$ and $\chi(s)=1$ for $s\geq2$, and extend
\[ b^{(j)}(t,x)\triangleq\begin{cases}\chi(jt)b(t,x),&0<t\leq T,\\0,&t\leq0.\end{cases} \]
Let $\rho\in C_c^\infty((0,1))$ and $\zeta\in C_c^\infty(B_1)$ be nonnegative mollifiers of integral one, put $\rho_\varepsilon(s)=\varepsilon^{-1}\rho(s/\varepsilon)$ and $\zeta_\varepsilon(x)=\varepsilon^{-2}\zeta(x/\varepsilon)$, and choose $0<\varepsilon_j<(4j)^{-1}$ with $\varepsilon_j\downarrow0$. Define, for $0\leq t\leq T$,
\[ b_j(t,x)\triangleq\int_0^{\varepsilon_j}\rho_{\varepsilon_j}(s)\bigl(\zeta_{\varepsilon_j}*b^{(j)}(t-s,\cdot)\bigr)(x)\,ds. \]
If $b\in B_T$ for every $T>0$, use the first branch in the definition of $b^{(j)}$ for all $t>0$ and the same formula for $b_j$ on all $t\geq0$; one sequence then works simultaneously on every finite time interval.
Each $b_j$ is smooth up to $t=0$, divergence free, and vanishes for $0\leq t\leq j^{-1}$. Moreover, it satisfies the uniform bound
\[ \sup_j\norm{b_j}_{B_T}\leq C_{a,\eta}\norm{b}_{B_T}. \]
The cutoff and mollification imply $b_j\to b$ locally uniformly on $(0,T]\times\R^2$. We next choose the data approximation explicitly. Truncate $f$ at heights $\pm j$, multiply it by a smooth cutoff that equals one on $B_j$ and is supported in $B_{2j}$, and convolve the result with a standard mollifier at length scale $j^{-1}$. This gives $f_j\in C_c^\infty$ such that $f_j\weakstar f$ in $L^{p_c,\infty}$ and $\sup_j\norm{f_j}_{L^{p_c,\infty}}\leq C_aG$. The truncation and multiplication do not increase the distribution function, and convolution is bounded on $L^{p_c,\infty}$. Moreover, distributional convergence together with the uniform bound gives weak-star convergence by density of test functions in $L^{p_c',1}$. Let $\vartheta_j$ solve $\partial_t\vartheta_j+b_j\cdot\nabla\vartheta_j+\Lambda^{1+a}\vartheta_j=0$ with $\vartheta_j(0)=f_j$. It is standard that $\norm{\vartheta_j(t)}_{L^r} \leq \norm{f_j}_{L^r}$ for every $1\leq r\leq\infty$ and $t > 0$. Real interpolation of the solution operator gives
\[ \sup_j\sup_{0<t\leq T}\norm{\vartheta_j(t)}_{L^{p_c,\infty}}\leq C_aG. \]
Since the fractional heat semigroup preserves $C_0$ and the drift $b_j$ is bounded on every finite time interval, the classical solution $\vartheta_j$ belongs to $C_0(\R^2)$ at every positive time. When the maximum $M_j(t)=\sup_{x\in\R^2} (\vartheta_j)(t,x)$ is positive, then it is attained. At a maximizing point the drift term vanishes, and \cref{lem:fractional-maximum}, together with the preceding weak-$L^{p_c}$ bound, gives the upper-Dini-derivative inequality
\[ \frac{d}{dt} M_j(t)+c_a\frac{M_j(t)^2}{G}\leq0. \]
Comparison with the ordinary differential equation yields $M_j(t)\leq C_aGt^{-1}$. Applying the same argument to $-\vartheta_j$ gives $\norm{\vartheta_j(t)}_{L^\infty}\leq C_aGt^{-1}$.

\noindent\emph{Duhamel's formula and positive-time regularity.}\par
Let $\vartheta$ be any distributional solution satisfying \eqref{eq:linear-critical-bounds}. For any $\phi\in\mathcal S(\R^2)$, testing the equation with $S(t-r)\phi$ gives
\[ \frac d{dr}\pair{\vartheta(r)}{S(t-r)\phi}=\pair{b\vartheta(r)}{\nabla S(t-r)\phi}. \]
On every compact interval $[s,t] \Subset(0,T]$, the right side belongs to $L^1(s,t)$, since $b\vartheta\in L^\infty([s,t] \times \R^2)$. So, $r \mapsto \langle \pair{\vartheta(r)}{S(t-r)\phi}$ is absolutely continuous and integration in $r$ on the interval $[s,t]$ proves the mild formulation \eqref{eq:linear-variation-constants}.

We now make the H\"older estimate explicit. Fix $0<\tau<T$ and set
\[ M_{\tau,T}\triangleq\sup_{\tau/2\leq t\leq T}\norm{b(t)\vartheta(t)}_{L^\infty}\leq C_aG\norm{b}_{B_T}\left(\frac{\tau}{2}\right)^{-\frac{1+2a}{1+a}}. \]
The bounds for $b_j$ and $\vartheta_j$ give the same estimate uniformly in $j$.
For $0<t\leq T$, the fractional heat kernel satisfies
\[ \norm{S(t)g}_{C^\beta}\leq C_{a,\beta,T}t^{-\frac{\beta}{1+a}}\norm g_{L^\infty},\qquad \norm{\nabla S(t)g}_{C^\beta}\leq C_{a,\beta,T}t^{-\frac{1+\beta}{1+a}}\norm g_{L^\infty}. \]
Hence, for $0<\beta<a$ and $\tau\leq t\leq T$, \eqref{eq:linear-variation-constants} with $s=\tau/2$ gives
\[ \begin{aligned} \norm{\vartheta(t)}_{C^\beta}&\leq C_{a,\beta,T}(t- {\textstyle \frac{\tau}{2}})^{-\frac{\beta}{1+a}}\norm{\vartheta({\textstyle \frac{\tau}{2}})}_{L^\infty}+C_{a,\beta,T}M_{\tau,T}\int_{\tau/2}^t(t-r)^{-\frac{1+\beta}{1+a}}\,dr\\&\leq C_{a,\beta,\tau,T}G\bigl(1+\norm{b}_{B_T}\bigr), \end{aligned} \]
because $(1+\beta)/(1+a)<1$. If $0<\beta'<\beta<a$, $0<h\leq1$, and $t,t+h\in[\tau,T]$, the semigroup property and \eqref{eq:linear-variation-constants} give
\[ \vartheta(t+h)-\vartheta(t)=\bigl(S(h)-I\bigr)\vartheta(t)-\int_t^{t+h}\nabla S(t+h-r)\cdot\bigl(b(r)\vartheta(r)\bigr)\,dr. \]
Using $\norm{(S(h)-I)g}_{C^{\beta'}}\leq C h^{(\beta-\beta')/(1+a)}\norm g_{C^\beta}$ and the second heat-kernel estimate with $\beta'$ yields
\[ \norm{\vartheta(t+h)-\vartheta(t)}_{C^{\beta'}}\leq C h^{\frac{\beta-\beta'}{1+a}}\norm{\vartheta(t)}_{C^\beta}+C_{a,\beta'}M_{\tau,T}h^{\frac{a-\beta'}{1+a}}. \]
Thus the approximating solutions are uniformly bounded in $C^\beta$ and equicontinuous in time with values in $C^{\beta'}$ on every positive-time strip. Arzel\`a--Ascoli and a diagonal extraction yield local uniform convergence to a distributional solution $\vartheta$ on $(0,T]\times\R^2$, and the Lorentz space bounds pass to the limit.

\noindent\emph{Initial trace.}\par
For every approximating solution, integration of its weak equation from $0$ to $t$ gives, uniformly in $j$ and for every $\phi\in C_c^\infty(\R^2)$,
\[ \begin{aligned} \abs{\pair{\vartheta_j(t)-f_j}{\phi}}&\leq C_a G\int_0^t\left(\norm{\Lambda^{1+a}\phi}_{L^{p_c',1}}+\norm{b}_{B_T} \tau^{-\frac{a}{1+a}}\norm{\nabla\phi}_{L^{p_c',1}}\right)\,d\tau\\&\leq C_aG\left(t\norm{\Lambda^{1+a}\phi}_{L^{p_c',1}}+(1+a)\norm{b}_{B_T}t^{\frac1{1+a}}\norm{\nabla\phi}_{L^{p_c',1}}\right). \end{aligned} \]
Passing to the locally uniform limit on the compact support of $\phi$ and using $f_j\weakstar f$ gives
\[ \abs{\pair{\vartheta(t)-f}{\phi}}\leq C_aG\left(t\norm{\Lambda^{1+a}\phi}_{L^{p_c',1}}+(1+a)\norm{b}_{B_T}t^{\frac1{1+a}}\norm{\nabla\phi}_{L^{p_c',1}}\right) \leq C_{a,\norm{b}_{B_T},\phi}G(t + t^{\frac1{1+a}}). \]
Moreover, for any $g\in L^{p_c',1}$, the uniform weak-$L^{p_c}$ bound and density of $C_c^\infty$ give
\[ \limsup_{t\downarrow0}\abs{\pair{\vartheta(t)-f}{g}}\leq(C_a+1)G\inf_{\phi\in C_c^\infty}\norm{g-\phi}_{L^{p_c',1}}=0, \]
which proves the initial trace in assertion~(a). If $\vartheta$ is any other solution in the class of assertion~(b), integrate its weak equation from $\delta$ to $t$ and let $\delta\downarrow0$ using its weak-star initial trace; the same calculation proves \eqref{eq:linear-trace-rate}.

\noindent\emph{Stability and the global construction.}\par
Consider a sequence as in assertion~(d). On every compact strip $[\tau,T]\times B_R$, the preceding spatial H\"older bound and time modulus give a uniformly convergent subsequence. A diagonal extraction produces a locally uniform limit $\vartheta$ on $(0,T]\times\R^2$. Since $b_n\to b$ locally uniformly, the linear equations pass to the limit against a test function $\Phi \in C^\infty_c((0,T] \times \R^2)$. Notice that the contribution from the fractional dissipation term $\Lambda^{1+a} \Phi$ is not compactly supported, but integrable since  $|\Lambda^{1+a} \Phi| \leq C(1+|x|)^{-2-a}$. The estimate \eqref{eq:linear-trace-rate} with uniformly bounded $\norm{b_n}_{B_T}$ shows that the limit has initial datum $f$. Local uniform convergence handles the compact part, and the common $L^{p_c,\infty}$ or $L^\infty$ bound makes the tail uniformly small. Thus the convergence is weak-star in both spaces, proving assertion~(d).

For assertion~(e), use the global approximants constructed above. Their bounds are uniform on every finite interval because $\sup_{T>0}\norm{b}_{B_T}<\infty$. Extract successively on $[m^{-1} ,m]\times B_m$ and take a diagonal subsequence. The resulting distributional solution is defined on $t \in (0,\infty)$ and satisfies \eqref{eq:linear-critical-bounds} for all $t > 0$, which proves assertion~(e).
\end{proof}

\begin{proof}[Proof of Theorem \ref{thm:critical-profile}]
Homogeneity gives $\omega_0(r,\theta)=r^{-(1+a)}\overline\omega_0(\theta)$ and therefore
\[ \abs{\{\abs{\omega_0}>\lambda\}}= \int_{\Sph^1} \int_0^{(|\overline\omega_0(\theta)|/\lambda)^{\frac{1}{1+a}}} r \,dr \,d\theta = \frac12\lambda^{-p_c}\int_{\Sph^1}\abs{\overline\omega_0(\theta)}^{p_c}\,d\theta. \]
Thus $\omega_0\in L^{p_c,\infty}$, and the analogous calculation gives $v_0\in L^{q_c,\infty}$. The flux and circulation of $v_0$ on $\partial B_r$ are $O(r^{1-a})$, so integration by parts on $\R^2\setminus B_r$ and passage to the limit $r\downarrow0$ show that $\diver v_0=0$ and $\curl v_0=\omega_0$ on all of $\R^2$. The field $K*\omega_0$ has the same divergence and curl and belongs to $L^{q_c,\infty}$. Their difference is therefore an entire harmonic vector field in $L^{q_c,\infty}$, and the mean-value estimate forces it to vanish. Hence $v_0=K*\omega_0$. Put $G=\norm{\omega_0}_{L^{p_c,\infty}}$. If $G=0$, then $v_0=0$ and the theorem holds with $W=Q=0$; hence assume $G>0$.

\medskip\noindent\emph{Step 1. Function space.}\par
Use the associate norm
\[ \norm{f}_{(p_c,\infty)}\triangleq\sup_{\norm{g}_{L^{p_c',1}}\leq1}\abs{\int_{\R^2}fg}, \]
which is equivalent to the weak-$L^{p_c}$ quasi-norm. Choose $G_*$ and $A_*$ large enough that the time-one estimates in \eqref{eq:linear-critical-bounds} imply $\norm{\vartheta(1)}_{(p_c,\infty)}\leq G_*$ and $\norm{\vartheta(1)}_{L^\infty}\leq A_*$ whenever the datum is $\omega_0$, and define
\[ \mathcal K\triangleq\{\Xi\in L^{p_c,\infty}\cap L^\infty:\ \norm\Xi_{(p_c,\infty)}\leq G_*,\ \norm\Xi_{L^\infty}\leq A_*\}. \]
Equip $\mathcal K$ with simultaneous weak-star convergence in $L^{p_c,\infty}=(L^{p_c',1})^*$ and $L^\infty=(L^1)^*$. This is a compact metrizable topology on $\mathcal K$: the two preduals are separable, every sequence has a subsequence converging in both weak-star topologies, the two limits agree on $C_c^\infty$, and the norm bounds pass to the common limit. The set $\mathcal K$ is also convex.

\medskip\noindent\emph{Step 2. Iteration scheme.}\par
For $\Xi\in\mathcal K$, set $V_\Xi=K*\Xi$ and
\[ b_\Xi(t,x)\triangleq t^{-\frac{a}{1+a}}V_\Xi(t^{-\frac{1}{1+a}}x). \]
Fix $\eta<\eta_0<1$ and choose $r>\max\{p_c, \frac{2}{1-\eta_0}\}$. Interpolation between the two bounds defining $\mathcal K$ gives a uniform $L^r$ bound for $\Xi$, and \cref{lem:biot-aux} then implies, uniformly in $\Xi$,
\[ \norm{V_\Xi}_{L^\infty}+[V_\Xi]_{\dot C^{\eta_0}}+\norm{\nabla V_\Xi}_{L^{p_c,\infty}}\leq C_{a,\eta_0,G_*,A_*}. \]
Since $V_\Xi\in C_b^{\eta_0}\cap L^{q_c,\infty}$, it tends to zero at infinity, and its dilations are continuous in $L^\infty$. Interpolation between $C^0$ and $C^{\eta_0}$ gives continuity in $C^\eta$. Moreover,
\[ t^{\frac{a}{1+a}}\norm{b_\Xi(t)}_{L^\infty}=\norm{V_\Xi}_{L^\infty},\qquad t^{\frac{\eta + a}{1+a}}[b_\Xi(t)]_{\dot C^\eta}=[V_\Xi]_{\dot C^\eta},\qquad \norm{\nabla b_\Xi(t)}_{L^{p_c,\infty}}=\norm{\nabla V_\Xi}_{L^{p_c,\infty}}. \]
Thus $b_\Xi\in B_T$ for every $T>0$, with $\norm{b_\Xi}_{B_T}\leq B_*$ uniformly in $\Xi$ and $T$. Choose a global solution $\vartheta_\Xi$ supplied by Lemma~\ref{lem:linear-drift}(a) and (e):
\begin{equation} \label{eq:self-similar-linear}
\partial_t\vartheta_\Xi+b_\Xi\cdot\nabla\vartheta_\Xi+\Lambda^{1+a}\vartheta_\Xi=0,\qquad \vartheta_\Xi(0)=\omega_0.
\end{equation}
It satisfies \eqref{eq:linear-critical-bounds}, and the integrated weak equation gives, for $\phi\in C_c^\infty$ and $t>0$,
\[ \abs{\pair{\vartheta_\Xi(t)-\omega_0}{\phi}}\leq C_aG\left(t\norm{\Lambda^{1+a}\phi}_{L^{p_c',1}}+\alpha B_*t^{\frac{1}{1+a}}\norm{\nabla\phi}_{L^{p_c',1}}\right). \]

The linear solution is not known to be unique in this critical class, so we impose self-similarity by averaging its scaled copies. For $s\in\R$, define
\[ (\mathcal D_s\vartheta)(t,x)\triangleq e^{(1+a) s}\vartheta(e^{(1+a) s}t,e^sx). \]
The scaling identities for $b_\Xi$ and $\omega_0$ show that $\mathcal D_s\vartheta_\Xi$ solves the same equation. Its critical bounds and trace estimate are independent of $s$: after changing variables in the trace pairing, the test function is $\phi_s(y)=e^{(a-1)s}\phi(e^{-s}y)$, and
\[ \norm{\Lambda^\alpha\phi_s}_{L^{p_c',1}}=e^{-\alpha s}\norm{\Lambda^\alpha\phi}_{L^{p_c',1}},\qquad \norm{\nabla\phi_s}_{L^{p_c',1}}=e^{-s}\norm{\nabla\phi}_{L^{p_c',1}}. \]
For $N>0$, define the average by:
\[ \pair{\vartheta_{\Xi,N}(t)}{g}\triangleq\frac1{2N}\int_{-N}^N\pair{\mathcal D_s\vartheta_\Xi(t)}{g}\,ds. \]
The integrand $s\mapsto \pair{\mathcal D_s\vartheta_\Xi(t)}{g}$ is continuous for all test functions $g\in C^\infty_c$. The uniform bound of $\mathcal D_s\vartheta_\Xi(t)$ in $L^{p_c,\infty}$ gives that the integrand is measurable for any $g \in L^{p_c',1}$. Thus, the definition above defines the average $\vartheta_{\Xi,N}(t) \in L^{p_c,\infty}$.

By linearity, $\vartheta_{\Xi,N}$ solves \eqref{eq:self-similar-linear} and satisfies the same bounds and the initial trace estimate. Lemma~\ref{lem:linear-drift}(c) makes these averages uniformly bounded and equicontinuous on every positive-time compact set. Thus a subsequence converges locally uniformly to a solution $\overline\vartheta_\Xi$. The nonlocal term passes to the limit by pairing with $\Lambda^\alpha\phi\in L^{p_c',1}$ and splitting that function into a compact part and a small Lorentz tail. For fixed $h\in\R$,
\[ \mathcal D_h\vartheta_{\Xi,N}-\vartheta_{\Xi,N}=\frac1{2N}\left(\int_N^{N+h}\mathcal D_s\vartheta_\Xi\,ds-\int_{-N}^{-N+h}\mathcal D_s\vartheta_\Xi\,ds\right), \]
where integrals over reversed intervals have their usual oriented meaning. Consequently, for every $g\in L^{p_c',1}$,
\[ \abs{\pair{\mathcal D_h\vartheta_{\Xi,N}(t)-\vartheta_{\Xi,N}(t)}{g}}\leq C_aG\frac{\abs h}{N}\norm{g}_{L^{p_c',1}}\longrightarrow0. \]
Local uniform convergence of $\vartheta_{\Xi,N}$ on positive-time compact sets also implies local uniform convergence of their fixed dilates $\mathcal D_h\vartheta_{\Xi,N}$. Passing to the limit in the preceding pairing gives $\mathcal D_h\overline\vartheta_\Xi=\overline\vartheta_\Xi$ in distributions, hence pointwise because both sides are continuous. With $\Theta=\overline\vartheta_\Xi(1)$ and $h=-\alpha^{-1}\log t$, this gives
\begin{equation} \label{eq:linear-self-similar}
\overline\vartheta_\Xi(t,x)=t^{-1}\Theta(t^{-1/\alpha}x).
\end{equation}
Let $\mathfrak T(\Xi)$ be the set of the time-one profiles of all self-similar solutions of \eqref{eq:self-similar-linear} with the uniform bounds \eqref{eq:linear-critical-bounds} and initial trace estimate \eqref{eq:linear-trace-rate}. The construction shows that this set is nonempty. It is contained in $\mathcal K$ and is convex because the equation is linear for fixed $\Xi$.

\medskip\noindent\emph{Step 3. Fixed-point theorem.}\par
We use the following form of the Kakutani fixed-point theorem (see Corollary 16.51 in \cite{AliprantisBorder1999}): Suppose $C$ is a compact convex subset of a locally convex Hausdorff space, $F(X)$ is a nonempty compact subset of $C$ for each $X\in C$, and the graph of $F$ is closed. Then there exists $X\in C$ with $X\in F(X)$.

Since $\mathcal{K}$ is a compact convex subset with respect to a metrizable topology, it suffices to verify the closedness of the graph of $\mathfrak T$ and the compactness of $\mathfrak T(\Xi)$. Suppose $\Xi_n\to\Xi$ in $\mathcal K$, $\Theta_n\in\mathfrak T(\Xi_n)$, and $\Theta_n\to\Theta$ in $\mathcal K$. The uniform $L^{p_c,\infty}\cap L^\infty$ bounds give a uniform $L^r$ bound for every $r>p_c$. Taking $r>2$, the Morrey inequality makes $V_{\Xi_n}=K*\Xi_n$ has a subsequence that converges to $\widetilde{V}$ locally unformlly. Then $\diver\widetilde V=0$ and $\curl\widetilde V=\Xi$ in distributions. Moreover, the uniform $L^{q_c,\infty}$ bound passes to $\widetilde V$. Thus $\widetilde V-K*\Xi$ is an entire harmonic vector field in $L^{q_c,\infty}$ and is therefore zero. Thus
\[ V_{\Xi_n}\longrightarrow V_\Xi\quad\text{locally uniformly},\qquad b_{\Xi_n}\longrightarrow b_\Xi\quad\text{locally uniformly}. \]
Lemma~\ref{lem:linear-drift}(d) now gives, after passing to a subsequence, a solution $\vartheta$ with drift $b_\Xi$ such that $\vartheta_n\to\vartheta$ locally uniformly on positive-time compact sets and weak-star in $L^{p_c,\infty}$ at every positive time. The initial trace estimate \eqref{eq:linear-trace-rate}, self-similar identity $\vartheta_n (t,x) = t^{-1} \Theta_n (t^{-1/\alpha} x)$, and equality $\vartheta_n(1)=\Theta_n$ pass to the limit, so $\vartheta(1) = \Theta\in\mathfrak T(\Xi)$. Hence the graph is closed. All hypotheses of the fixed-point theorem are satisfied, so there is $\Omega\in\mathcal K$ with $\Omega\in\mathfrak T(\Omega)$.

Set $W=K \ast \Omega$ and
\[ \omega(t,x)=t^{-1}\Omega(t^{-\frac{1}{1+a}}x),\qquad v(t,x)=K*\omega(t,x)=t^{-\frac{a}{1+a}}W(t^{-\frac{1}{1+a}}x)=b_\Omega(t,x). \]
Then
\[ \partial_t\omega+v\cdot\nabla\omega+\Lambda^\alpha\omega=0 \]
in distributions on $(0,\infty)\times\R^2$. Substitution of the self-similar formulas is legitimate in distributions because $W$ and $\Omega$ are bounded, and it gives
\begin{equation} \label{eq:unit-vorticity-profile}
\Lambda^{1+a}\Omega+\left(W-\frac{y}{1+a}\right)\cdot\nabla\Omega-\Omega=0.
\end{equation}
The time-one critical estimate gives \eqref{eq:unit-critical-bound}, and the common trace estimate gives $\omega(t)\weakstar\omega_0$ in $L^{p_c,\infty}$ as $t\downarrow0$.

\medskip\noindent\emph{Step 4. Smoothness, pressure, and the velocity trace.}\par
Lemma~\ref{lem:linear-drift}(c) gives $\Omega\in C^\beta$ for every $0<\beta<a$. If $\Omega\in C^s\cap L^{p_c,\infty}\cap L^\infty$, the standard singular-integral estimate for $\nabla W$ and the low-frequency bound supplied by $W\in L^\infty$ give $W\in C^{s+1}$. Hence $W\Omega\in C^s$. At time one, \eqref{eq:linear-variation-constants} reads
\[ \Omega=S({\textstyle\frac{1}{2}})\omega({\textstyle\frac{1}{2}})-\int_{\frac{1}{2}}^1\nabla S(1-\tau)\cdot\bigl(v(\tau)\omega(\tau)\bigr)\,d\tau. \]
On $\tau \in [\frac{1}{2},1]$, self-similarity gives a uniform $C^s$ bound for $v(\tau)\omega(\tau)$. For every $0<\sigma<a$, the fractional heat kernel satisfies
\[ \norm{\nabla S(t)F}_{C^{s+\sigma}}\leq C_{a,s,\sigma}t^{-\frac{1+\sigma}{1+a}}\norm F_{C^s}. \]
Since the exponent $\frac{1+\sigma}{1+a}$ is less than $1$, the integral is finite in $C^{s+\sigma}$. The first term is smooth. Therefore $\Omega\in C^{s+\sigma}$. Repeating the argument while choosing noninteger intermediate exponents reaches arbitrarily large exponents; hence $\Omega\in C^\infty$, and then $W\in C^\infty$ as well.

Define $Q=\sum_{i,j}R_iR_j(W_iW_j)$. Since $W\in L^{q_c,\infty}$ and $q_c/2=1/a$, the weak product estimate and the Calder\'on--Zygmund bound give $Q\in L^{\frac{1}{a},\infty}$. Also, $\Delta Q=-\sum_{i,j}\partial_i\partial_j(W_iW_j) \in C^\infty$ and thus $Q\in C^\infty$. Lemma~\ref{lem:pressure-normalization} with $\epsilon=1$ now gives \eqref{eq:unit-dissipative-profile}. Scaling that equation shows that $(v,q)$, with $q(t,x)=t^{-\frac{2a}{1+a}}Q(t^{-\frac{1}{1+a}}x)$, solves the hypodissipative velocity equation for $t>0$.

Finally, if $g\in L^{q_c',1}$, the Biot--Savart estimate gives $K \ast g\in L^{p_c',1}$, and therefore
\[ \pair{v(t)-v_0}{g}=\pair{K*(\omega(t)-\omega_0)}{g}=\pair{\omega(t)-\omega_0}{-K \ast g}\longrightarrow0. \]
This shows $v(t)\weakstar v_0$ in $L^{q_c,\infty}$.
\end{proof}

\begin{lemma} \label{lem:pressure-normalization}
Let $\epsilon>0$ and suppose that $W$ is smooth, divergence free, $W\in L^{q_c,\infty}$, $\Omega=\curl W$, and
\[ \epsilon\Lambda^{1+a}\Omega+\left(W-\frac{y}{1+a}\right)\cdot\nabla\Omega-\Omega=0. \]
Set $Q=\sum_{i,j}R_iR_j(W_iW_j)$. Then $Q\in L^{\frac{1}{a},\infty}$ and
\begin{equation} \label{eq:normalized-profile}
\epsilon\Lambda^{1+a}W-\frac1{1+a}\bigl(aW+y\cdot\nabla W\bigr)+W\cdot\nabla W+\nabla Q=0
\end{equation}
in distributions.
\end{lemma}

\begin{proof}
Let $F$ be the left side of \eqref{eq:normalized-profile}, then $F$ is harmonic. For radial $\Phi\in C_c^\infty(\R^2;\R^2)$ put $\Phi_{x_0, R}(y)=\Phi((y-x_0)/R)$. Then, by integrating in parts,
\[ \abs{\pair{W}{\Phi_{x_0, R}}}+\abs{\pair{y\cdot\nabla W}{\Phi_{x_0, R}}}\leq C_\Phi\norm W_{L^{q_c,\infty}}R^{2-a}. \]
Furthermore,
\[ \abs{\pair{\Lambda^{1+a}W}{\Phi_{x_0, R}}}+\abs{\pair{\diver(W\otimes W)+\nabla Q}{\Phi_{x_0, R}}}\leq C_{\Phi,\epsilon,W}R^{1-2a}. \]
Thus $\abs{\pair{F}{\Phi_{x_0, R}}}=O(R^{2-a})=o(R^2)$. By the mean value theorem for harmonic functions, letting $R \to \infty$ gives $F \equiv 0$.
\end{proof}

\subsection{The approximating sequence of profiles and uniform estimates}

\begin{proposition}\label{prop:eps-family}
For every $\epsilon\in(0,1]$, there is a smooth solution $(U_\epsilon,P_\epsilon)$ of \eqref{eq:eps-profile}. If $\Omega_\epsilon=\curl U_\epsilon$, then $U_\epsilon=K*\Omega_\epsilon$, and the physical self-similar solution associated with this profile has weak-star initial trace $u_0$ in $L^{q_c,\infty}$ and vorticity trace $\omega_0$ in $L^{p_c,\infty}$.
\end{proposition}

\begin{proof}
Put $\ell_\epsilon=\epsilon^{1/(1+a)}$ and apply \cref{thm:critical-profile} to $v_{0,\epsilon}=\epsilon^{-1}u_0$. If $W_\epsilon,Q_\epsilon$ are the resulting unit-dissipation profiles, define
\[ U_\epsilon(y)=\ell_\epsilon W_\epsilon(y/\ell_\epsilon),\qquad P_\epsilon(y)=\ell_\epsilon^2Q_\epsilon(y/\ell_\epsilon). \]
Since $\ell_\epsilon^{1+a}=\epsilon$, direct scaling gives \eqref{eq:eps-profile}; it also preserves the Biot--Savart identity. If $v_\epsilon$ denotes the unit-dissipation physical solution, then the coefficient-$\epsilon$ solution is $u_\epsilon(t,x)=\epsilon v_\epsilon(\epsilon t,x)$. Its weak-star trace is therefore $\epsilon v_{0,\epsilon}=u_0$, and the same calculation for curls gives the vorticity trace.
\end{proof}

\begin{proposition} \label{prop:vort-bound}
There is a constant depending only on $a$ and $u_0$ such that, for every $\epsilon\in(0,1]$,
\begin{equation} \label{eq:vort-bound}
\norm{\Omega_\epsilon}_{L^{p_c,\infty}}\leq C(a)\norm{\omega_0}_{L^{p_c,\infty}}.
\end{equation}
Moreover,
\begin{equation} \label{eq:velocity-bounds}
\norm{U_\epsilon}_{L^{q_c,\infty}}+\norm{\nabla U_\epsilon}_{L^{p_c,\infty}}\leq C(a,u_0),
\end{equation}
and
\begin{equation} \label{eq:pressure-bound}
\norm{P_\epsilon}_{L^{\frac{1}{a},\infty}}\leq C(a,u_0).
\end{equation}
\end{proposition}

\begin{proof}
If $\Xi_\epsilon=\curl W_\epsilon$, scaling and \eqref{eq:unit-critical-bound} give
\[ \norm{\Omega_\epsilon}_{L^{p_c,\infty}}=\ell_\epsilon^{2/p_c}\norm{\Xi_\epsilon}_{L^{p_c,\infty}}\leq C_a\ell_\epsilon^{1+a}\epsilon^{-1}\norm{\omega_0}_{L^{p_c,\infty}}=C_a\norm{\omega_0}_{L^{p_c,\infty}}. \]
The Lorentz Hardy--Littlewood--Sobolev and Calder\'on--Zygmund estimates give
\begin{equation} \label{eq:HLS}
\norm{K*f}_{L^{q_c,\infty}}\leq C_a\norm f_{L^{p_c,\infty}}.
\end{equation}
\begin{equation} \label{eq:CZ}
\norm{\nabla(K*f)}_{L^{p_c,\infty}}\leq C_a\norm f_{L^{p_c,\infty}}.
\end{equation}
Applying these estimates to $U_\epsilon=K*\Omega_\epsilon$ proves \eqref{eq:velocity-bounds}. Finally, $P_\epsilon=\sum R_iR_j(U_{\epsilon,i}U_{\epsilon,j})$ by its construction and scaling. The weak product estimate and Calder\'on--Zygmund boundedness on $L^{\frac{1}{a},\infty}$ prove \eqref{eq:pressure-bound}.
\end{proof}

\section{Compactness and passage to Euler}

\begin{proposition}\label{prop:compactness}
There exist a sequence $\epsilon_n\downarrow0$, vector fields $U,\Omega$, and a scalar field $P$ such that
\begin{align}
 & U_{\epsilon_n}\longrightarrow U
 &&\text{strongly in }L^2_{\rm loc}(\R^2),\label{eq:strong-U}\\
 & \Omega_{\epsilon_n}\weakstar\Omega
 &&\text{in }L^{p_c,\infty}(\R^2),\label{eq:weak-Omega}\\
 & P_{\epsilon_n}\weakstar P
 &&\text{in }L^{\frac{1}{a},\infty}(\R^2),\label{eq:weak-pressure}\\
 & \Omega=\curl U,\qquad U=K*\Omega.&&\label{eq:curl-limit}
\end{align}
The limit satisfies
\[ U\in L^{q_c,\infty}(\R^2),\qquad \nabla U\in L^{p_c,\infty}(\R^2),\qquad P\in L^{\frac{1}{a},\infty}(\R^2). \]
\end{proposition}

\begin{proof}
Fix $r$ with
\begin{equation} \label{eq:r-choice}
1<r<p_c.
\end{equation}
On every ball $B_R$, the bounds \eqref{eq:velocity-bounds} give
\[ \sup_{\epsilon\in(0,1]} \norm{U_\epsilon}_{W^{1,r}(B_R)}<\infty. \]
Indeed, the passage from weak $L^{q_c}$ and weak $L^{p_c}$ to $L^r(B_R)$ follows from \cref{prop:lorentz-basic}(iii), since $r<p_c<q_c$.

In dimension two, $W^{1,r}(B_R)\Subset L^2(B_R)$ whenever $r>1$, since $2<\frac{2r}{2-r}$. A diagonal extraction proves \eqref{eq:strong-U}. Applying Banach--Alaoglu to the separable Lorentz preduals gives, after a further common extraction, weak-star limits for $U_\epsilon$, $\nabla U_\epsilon$, $\Omega_\epsilon$, and $P_\epsilon$ in their respective spaces. The local strong limit identifies the first two limits with $U$ and $\nabla U$, proving all asserted bounds and \eqref{eq:weak-pressure}. Passing to the distributional curl gives $\Omega=\curl U$. Finally, $U_\epsilon=K*\Omega_\epsilon$ passes to distributions because the adjoint Hardy--Littlewood--Sobolev estimate maps $L^{q_c',1}$ to $L^{p_c',1}$; hence $U=K*\Omega$.
\end{proof}

\begin{proposition}\label{prop:profile-limit}
The limit pair $(U,P)$ solves
\[ -\frac{1}{1+a}\bigl(aU+y\cdot\nabla U\bigr)+U\cdot\nabla U+\nabla P=0,\qquad \diver U=0, \]
in distributions on $\R^2$.
\end{proposition}

\begin{proof}
Test \eqref{eq:eps-profile} against an arbitrary $\varphi\in C_c^\infty(\R^2;\R^2)$. Integration by parts gives
\begin{equation} \label{eq:eps-weak-profile}
0=\epsilon\int_{\R^2}U_\epsilon\cdot\Lambda^{1+a}\varphi\,dy+\frac{1}{1+a}\int_{\R^2}U_\epsilon\cdot\bigl((2-a)\varphi+y\cdot\nabla\varphi\bigr)\,dy-\int_{\R^2}(U_\epsilon\otimes U_\epsilon):\nabla\varphi\,dy-\int_{\R^2}P_\epsilon\diver\varphi\,dy.
\end{equation}
Here the coefficient $2-a$ follows from
\[ -\int(y\cdot\nabla U_\epsilon)\cdot\varphi =\int U_\epsilon\cdot(2\varphi+y\cdot\nabla\varphi). \]
Since $1<q_c<\infty$, the Lorentz--H\"older inequality in \cref{prop:lorentz-basic}(i), \cref{lem:lorentz-test-functions}, and \eqref{eq:velocity-bounds} imply
\[ \left| \epsilon\int_{\R^2}U_\epsilon\cdot\Lambda^{1+a}\varphi\,dy \right| \leq C\epsilon \norm{U_\epsilon}_{L^{q_c,\infty}} \norm{\Lambda^{1+a}\varphi}_{L^{q_c',1}} \longrightarrow0. \]
Since \eqref{eq:strong-U} implies $U_{\epsilon_n}\otimes U_{\epsilon_n}\to U\otimes U$ strongly in $L^1_{\rm loc}$, while \eqref{eq:weak-pressure} passes the pressure term to the limit, we may let $n\to\infty$ in \eqref{eq:eps-weak-profile}. The resulting identity for every compactly supported vector test field is precisely the asserted profile equation.
\end{proof}

We now prove the vorticity formulation in the range $0<a<\frac{1}{2}$.

\begin{proof}[Proof of \cref{cor:vorticity}]
Recall that
\[ \Omega\in L^{p_c,\infty}_{\rm loc}, \qquad U\in L^{q_c,\infty}_{\rm loc}. \]
By \cref{prop:lorentz-basic}(ii),
$U\Omega\in L^{r_*,\infty}_{\rm loc}$, where
\[ \frac1{r_*}=\frac1{p_c}+\frac1{q_c}. \]
The finite-measure embedding in \cref{prop:lorentz-basic}(iii) then gives $U\Omega\in L^1_{\rm loc}$ provided
\[ \frac1{p_c} + \frac1{q_c}<1. \]
This condition is $\frac{1 + a}{2} + \frac a2<1$, which is equivalent to $a<\frac{1}{2}$.

We identify the equation by passing directly to the viscous vorticity equation. Fix $\psi\in C_c^\infty(\R^2)$. Its weak form is
\begin{equation} \label{eq:weak-vort-eps}
0= \epsilon\int_{\R^2}\Omega_\epsilon\Lambda^{1+a}\psi\,dy-\int_{\R^2}\Omega_\epsilon U_\epsilon\cdot\nabla\psi\,dy +\frac1{1+a}\int_{\R^2}\Omega_\epsilon y\cdot\nabla\psi\,dy+(p_c-1)\int_{\R^2}\Omega_\epsilon\psi\,dy.
\end{equation}
Indeed, with $\widetilde U_\epsilon=U_\epsilon-y/(1+a)$, this follows by integrating $\widetilde U_\epsilon\cdot\nabla\Omega_\epsilon$ by parts and using $\diver\widetilde U_\epsilon=-p_c$.

Choose $r\in(1,p_c)$ sufficiently close to $p_c$. By local Sobolev compactness,
\[ U_{\epsilon_n}\to U \quad \text{strongly in }L^\rho_{\rm loc} \text{~~~for every~} \rho<\frac{2r}{2 - r}. \]
Letting $r\uparrow p_c$, the upper endpoint tends to $q_c$. Since $p_c'<q_c$ precisely when $a<\frac{1}{2}$, we may choose $\rho>p_c'$. On a bounded set, \cref{prop:lorentz-basic}(iii) gives the continuous embedding $L^\rho\hookrightarrow L^{p_c',1}$. Hence the convergence is strong in $L^{p_c',1}$, while $\Omega_{\epsilon_n}\weakstar\Omega$ in $L^{p_c,\infty}$. Therefore
\[ \Omega_{\epsilon_n}U_{\epsilon_n} \longrightarrow\Omega U \]
in distributions.

The dissipative term vanishes by
\cref{prop:lorentz-basic}(i), \cref{lem:lorentz-test-functions}, and the
uniform weak-$L^{p_c}$ bound:
\[ \left| \epsilon_n\int\Omega_{\epsilon_n}\Lambda^{1 + a}\psi \right| \leq C\epsilon_n \norm{\Omega_{\epsilon_n}}_{L^{p_c,\infty}} \norm{\Lambda^{1 + a}\psi}_{L^{p_c',1}} \longrightarrow0. \]
Passing to the limit in \eqref{eq:weak-vort-eps} gives
\begin{equation} \label{eq:weak-vort-limit}
\int_{\R^2}\Omega\left(U-\frac y{1+a}\right)\cdot\nabla\psi\,dy=(p_c-1)\int_{\R^2}\Omega\psi\,dy.
\end{equation}
Since
\[ \diver\left(U - \frac y{1 + a}\right) = - p_c, \]
\eqref{eq:weak-vort-limit} is exactly the distributional formulation of \eqref{eq:euler-vort-profile}. The weak formulation of the physical vorticity equation for $t>0$ follows from the self-similar change of variables. To include $t=0$, fix a nondecreasing cutoff $\theta\in C^\infty([0,\infty))$ such that $\theta=0$ on $[0,1]$ and $\theta=1$ on $[2,\infty)$. For any $\varphi\in C_c^\infty([0,\infty)\times\R^2)$, testing the positive-time equation with $\theta(t/\delta)\varphi(t,x)$ gives
\[ 0=\int_0^\infty\int_{\R^2}\theta(t/\delta)\omega\bigl(\partial_t\varphi+u\cdot\nabla\varphi\bigr)\,dx\,dt+\int_\delta^{2\delta}\frac{\theta'(t/\delta)}{\delta}\pair{\omega(t)}{\varphi(t)}\,dt. \]
For the first term,
\[ \abs{\int_0^\infty\int_{\R^2}(1-\theta(t/\delta))\omega\bigl(\partial_t\varphi+u\cdot\nabla\varphi\bigr)\,dx\,dt} \leq C_\varphi\int_0^{2\delta}\bigl(\norm{\omega}_{L^1(\supp_x\varphi)}+\norm{u\omega}_{L^1(\supp_x\varphi)}\bigr)\,dt\leq C_\varphi\delta. \]
For the second term, the weak-star initial trace in \cref{prop:trace} and Lorentz--H\"older give
\[ \abs{\pair{\omega(t)}{\varphi(t)}-\pair{\omega_0}{\varphi(0)}}\leq\abs{\pair{\omega(t)-\omega_0}{\varphi(0)}}+\norm{\omega(t)}_{L^{p_c,\infty}}\norm{\varphi(t)-\varphi(0)}_{L^{p_c',1}}\longrightarrow0, \quad \text{as}~t \downarrow 0. \]
Consequently,
\[ \int_\delta^{2\delta}\frac{\theta'(t/\delta)}{\delta}\pair{\omega(t)}{\varphi(t)}\,dt=\int_1^2\theta'(s)\pair{\omega(\delta s)}{\varphi(\delta s)}\,ds\longrightarrow\pair{\omega_0}{\varphi(0)}, \quad\text{as}~ \delta \downarrow 0. \]
Therefore, letting $\delta\downarrow0$ yields the vorticity weak formulation in \cref{def:weak}(b).
\end{proof}

\section{The initial trace and renormalization property}\label{sec:trace-renormalization}

For every $\epsilon\in(0,1]$, define the physical variables
\[ u_\epsilon(t,x)\triangleq t^{-\frac{a}{1+a}} U_\epsilon\left(\frac{x}{t^{{\frac{1}{1+a}}}}\right), \qquad p_\epsilon(t,x)\triangleq t^{-\frac{2a}{1+a}} P_\epsilon\left(\frac{x}{t^{{\frac{1}{1+a}}}}\right). \]
Then
\begin{equation} \label{eq:fractional-physical}
\partial_t u_\epsilon+u_\epsilon\cdot\nabla u_\epsilon+\nabla p_\epsilon+\epsilon\Lambda^{1+a}u_\epsilon=0,\qquad \diver u_\epsilon=0
\end{equation}
with initial datum $u_0$. Criticality of the exponents gives
\begin{equation} \label{eq:physical-uniform}
\begin{aligned} &\sup_{t>0}\left(\norm{u_\epsilon(t)}_{L^{q_c,\infty}}+\norm{\nabla u_\epsilon(t)}_{L^{p_c,\infty}}+\norm{p_\epsilon(t)}_{L^{{\frac{1}{a}},\infty}}\right)\\ &\hspace{5em}=\norm{U_\epsilon}_{L^{q_c,\infty}}+\norm{\nabla U_\epsilon}_{L^{p_c,\infty}}+\norm{P_\epsilon}_{L^{\frac{1}{a},\infty}}\leq C(a,u_0). \end{aligned}
\end{equation}

\begin{lemma}\label{lem:uniform-trace}
For every $\psi\in C_c^\infty(\R^2;\R^2)$,
\begin{equation} \label{eq:uniform-trace}
\abs{\pair{u_\epsilon(t)-u_0}{\psi}}\leq C(a,u_0,\psi)t,\qquad t>0,\quad \epsilon\in(0,1].
\end{equation}
\end{lemma}

\begin{proof}
The fixed-$\epsilon$ solution supplied by \cref{prop:eps-family} has weak-star initial datum $u_0$. Testing \eqref{eq:fractional-physical} on $(\tau,t)$ against the time-independent field $\psi$ and then sending $\tau\downarrow0$ gives
\begin{align*}
 \pair{u_\epsilon(t)-u_0}{\psi}
 &= \int_0^t\int_{\R^2} (u_\epsilon\otimes u_\epsilon):\nabla\psi\,dx\,d\tau +\int_0^t\int_{\R^2} p_\epsilon\diver\psi\,dx\,d\tau -\epsilon\int_0^t\int_{\R^2} u_\epsilon\cdot\Lambda^{1+a}\psi\,dx\,d\tau.
\end{align*}
Because $\frac{q_c}{2}={\frac{1}{a}}$, the weak-product estimate and the Lorentz--H\"older inequality in \cref{prop:lorentz-basic}(i)--(ii), together with \cref{lem:lorentz-test-functions} and \eqref{eq:physical-uniform}, yield
\[ \left| \int_{\R^2}(u_\epsilon\otimes u_\epsilon):\nabla\psi\,dx \right| \leq C\norm{u_\epsilon}_{L^{q_c,\infty}}^2 \norm{\nabla\psi}_{L^{\frac{1}{1-a},1}}. \]
Here $(1/a)'=1/(1-a)$, so the pressure term is bounded by the same Lorentz--H\"older inequality:
\[ \left| \int_{\R^2}p_\epsilon\diver\psi\,dx \right| \leq \norm{p_\epsilon}_{L^{{\frac{1}{a}},\infty}} \norm{\diver\psi}_{L^{\frac{1}{1-a},1}}. \]
Similarly, $q_c\in(1,\infty)$ and $\Lambda^{1+a}\psi\in L^{q_c',1}$ by \cref{lem:lorentz-test-functions}, so
\[ \left| \int_{\R^2}u_\epsilon\cdot\Lambda^{1+a}\psi\,dx \right| \leq \norm{u_\epsilon}_{L^{q_c,\infty}} \norm{\Lambda^{1+a}\psi}_{L^{q_c',1}}. \]
All three right-hand sides are uniform in time and $\epsilon$, which proves
\eqref{eq:uniform-trace}.
\end{proof}

\begin{proposition}\label{prop:trace}
Let
\[ u(t,x)\triangleq t^{-\frac{a}{1+a}} U\left(\frac{x}{t^{{\frac{1}{1+a}}}}\right), \qquad \omega(t,x)\triangleq \curl u(t,x) =t^{-1}\Omega\left(\frac{x}{t^{{\frac{1}{1+a}}}}\right). \]
Then $u\in C\bigl([0,\infty);L^2_{\rm loc}(\R^2)\bigr)$ and $\omega\in L^\infty\bigl((0,\infty); L^{p_c,\infty}(\R^2)\bigr)$. In particular,
\[ u(t)\to u_0 \quad\text{in }\mathcal D'(\R^2) \quad\text{and strongly in }L^2_{\rm loc}(\R^2) \quad\text{as }t\downarrow0,\]
and
\[ \omega(t)\weakstar\omega_0 \quad\text{in }L^{p_c,\infty}(\R^2) \quad\text{as }t\downarrow0. \]
Consequently, for every $1<r<p_c$,
\[ \omega(t)\rightharpoonup\omega_0 \quad\text{weakly in }L^r_{\rm loc}(\R^2)\quad\text{as }t\downarrow0. \]
\end{proposition}

\begin{proof}
For each fixed $t>0$, \eqref{eq:strong-U} and the self-similar change of variables imply
\[ u_{\epsilon_n}(t)\longrightarrow u(t) \quad\text{strongly in }L^2_{\rm loc}(\R^2). \]
Passing to the limit in \eqref{eq:uniform-trace} gives
\[ \abs{\pair{u(t)-u_0}{\psi}} \leq C(a,u_0,\psi)t \]
for every compactly supported $\psi$. This proves convergence in distributions.

It remains to upgrade the convergence. Choose $r$ as in \eqref{eq:r-choice}. The norms in
\[ u(t,x)=t^{-\frac{a}{1+a}}U(x/t^{{\frac{1}{1+a}}}),\qquad \nabla u(t,x)=t^{-1}\nabla U(x/t^{{\frac{1}{1+a}}}) \]
are scaling invariant in $L^{q_c,\infty}$ and $L^{p_c,\infty}$, respectively. Hence, for every $R>0$,
\[ \sup_{0<t\leq1}\norm{u(t)}_{W^{1,r}(B_R)} \leq C(R,a,u_0). \]
The set $\{u(t):0<t\leq1\}$ is therefore precompact in $L^2(B_R)$. Every sequence $t_j\downarrow0$ has a subsequence converging strongly in $L^2(B_R)$; its distributional limit must be $u_0$. Thus every such subsequence converges to $u_0$, and a standard contradiction argument shows that the full family converges strongly. Since $R$ is arbitrary, the asserted strong initial trace follows.

At each $t_0>0$, continuity of $t\mapsto u(t)$ in $L^2_{\rm loc}$ follows directly from the strong continuity of dilations and scalar multiplication on $L^2_{\rm loc}$. Combined with the continuity at $t=0$ proved above, this gives $u\in C\bigl([0,\infty);L^2_{\rm loc}(\R^2)\bigr)$.

The critical vorticity norm is invariant under the self-similar scaling:
\[ \norm{\omega(t)}_{L^{p_c,\infty}} =\norm{\Omega}_{L^{p_c,\infty}}. \]
Thus $\omega$ has the asserted uniform bound. Since $u(t)\to u_0$ strongly in $L^2_{\rm loc}$,
\[ \omega(t)=\curl u(t)\longrightarrow \omega_0 = \curl u_0 \quad\text{in }\mathcal D'(\R^2). \]
The uniform $L^{p_c,\infty}$ bound, the duality in \cref{prop:lorentz-basic}(iv), and the density of $C_c^\infty(\R^2)$ in $L^{p_c',1}(\R^2)$ upgrade this convergence to
\[ \omega(t)\weakstar\omega_0 \quad\text{in }L^{p_c,\infty}(\R^2). \]
On every bounded set, \cref{prop:lorentz-basic}(iii) gives $L^{p_c,\infty}\hookrightarrow L^r$ for $1<r<p_c$. Distributional convergence and weak compactness in $L^r$ therefore give the final assertion.
\end{proof}

\begin{proposition}\label{prop:renormalization}
The vorticity $\omega$ in \cref{prop:trace} is a renormalized solution with velocity $u$ and initial datum $\omega_0$ in the sense of \cref{def:weak}(c). Moreover, $\omega\in C([0,\infty);L^r_{\rm loc}(\R^2))$ with $\omega(0)=\omega_0$ for every $1\leq r<p_c$, and $\norm{\omega(t)}_{L^{p_c,\infty}}=\norm{\omega_0}_{L^{p_c,\infty}}$ for every $t>0$.
\end{proposition}

\begin{proof}
Fix $T>0$ and $1<r<p_c$, and let $\epsilon\downarrow0$ run along the subsequence in \cref{prop:compactness}.

\medskip\noindent\emph{Step 1. DiPerna--Lions theory.}\par\nopagebreak
The bounds in \cref{prop:compactness} and self-similar scaling give
\[ u\in L^\infty(0,T;W^{1,r}_{\rm loc}),\qquad \diver u=0,\qquad \frac{u}{1+|x|}\in L^1_tL^1_x+L^1_tL^\infty_x. \]
The last condition follows from $u\in L^\infty_tL^{q_c,\infty}_x$, $q_c>1$. Thus the DiPerna--Lions theory \cite{DiPernaLions1989} supplies a measure-preserving regular Lagrangian flow and the associated renormalized transport solutions. It also gives uniqueness for bounded solutions of the backward transport equation used below.

\medskip\noindent\emph{Step 2. The backward equation and its trace.}\par\nopagebreak
For $\chi\in C_c^\infty((0,T)\times\R^2)$, solve
\begin{equation} \label{eq:renorm-backward-eps}
-\partial_t\phi_\epsilon-u_\epsilon\cdot\nabla\phi_\epsilon+\epsilon\Lambda^{1+a}\phi_\epsilon=\chi,\qquad \phi_\epsilon(T)=0.
\end{equation}
Since $u_\epsilon$ is smooth for $t>0$, the backward equation defines $\phi_\epsilon$ on every interval $[\delta,T]$, $\delta>0$. Standard $L^p$ and energy estimates give, for every $1\leq p\leq\infty$,
\[ \norm{\phi_\epsilon(t)}_{L^p}\leq\int_t^T\norm{\chi(s)}_{L^p}\,ds,\qquad \epsilon\int_0^T\norm{\Lambda^{\frac{1+a}{2}}\phi_\epsilon(t)}_{L^2}^2\,dt\leq C_\chi. \]
For $\eta\in C_c^\infty(\R^2)$, Lorentz--H\"older and \eqref{eq:physical-uniform} give, uniformly in $0<\epsilon\leq1$,
\[ \abs{\frac{d}{dt}\pair{\phi_\epsilon}{\eta}}=\abs{\pair{\phi_\epsilon}{u_\epsilon\cdot\nabla\eta+\epsilon\Lambda^{1+a}\eta}-\pair{\chi}{\eta}}\leq C_{\chi,\eta}. \]
Here $\nabla\eta\in L^{q_c',1}$ and $\Lambda^{1+a}\eta\in L^1$. Together with the uniform $L^p$ bounds and density of smooth test functions in $L^{p'}$, this gives equicontinuity in the weak $L^p$ topology and weak traces $\phi_\epsilon(0)$ for every $1<p<\infty$.

For duality at the initial time, we need
\begin{equation} \label{eq:renorm-backward-trace}
\phi_\epsilon(t)\longrightarrow\phi_\epsilon(0)\quad\text{strongly in }L^{p_c',1}\quad\text{as }t\downarrow0.
\end{equation}
For fixed $\epsilon$, boundedness of the velocity profile gives $\norm{u_\epsilon(t)}_{L^\infty}\leq C_\epsilon t^{- \frac{a}{1+a}}$, so $u_\epsilon\in L^1_tL^\infty_x$. Let $\phi_{\epsilon,h}=\rho_h*\phi_\epsilon$ for a spatial mollifier $\rho_h$. The commutator
\[ R_h=u_\epsilon\cdot\nabla\phi_{\epsilon,h}-\rho_h*(u_\epsilon\cdot\nabla\phi_\epsilon) \]
converges to $0$ in $L^1_{t,x}$ as $h\downarrow0$. Indeed, incompressibility gives
\[ R_h(t,x)=\int_{\R^2}\nabla\rho_h(z)\cdot\bigl(u_\epsilon(t,x)-u_\epsilon(t,x-z)\bigr)\bigl(\phi_\epsilon(t,x-z)-\phi_\epsilon(t,x)\bigr)\,dz. \]
The $L^1$ and $L^\infty$ bounds imply $\phi_\epsilon\in L^1_tL^{p_c',1}_x$. Lorentz--H\"older and translation continuity in this space yield
\[ \norm{R_h}_{L^1_{t,x}}\leq C\norm{\nabla u_\epsilon}_{L^\infty_tL^{p_c,\infty}_x}\sup_{z\in\supp\rho_h}\norm{\phi_\epsilon(\cdot,\cdot-z)-\phi_\epsilon}_{L^1_tL^{p_c',1}_x}\longrightarrow0\qquad\text{as }h\downarrow0. \]
For fixed $h>0$, the mollified equation and $u_\epsilon\in L^1_tL^\infty_x$ give $\phi_{\epsilon,h}\in W^{1,1}(0,T;L^2)$ and justify spatial integration by parts. Moreover,
\[ \abs{\int_s^t\int_{\R^2}R_h\phi_{\epsilon,h}\,dx\,d\tau}\leq\norm{\phi_\epsilon}_{L^\infty_{t,x}}\norm{R_h}_{L^1_{t,x}}\longrightarrow0. \]
Testing $-\partial_t\phi_{\epsilon,h}-u_\epsilon\cdot\nabla\phi_{\epsilon,h}+\epsilon\Lambda^{1+a}\phi_{\epsilon,h}=\chi_h-R_h$, where $\chi_h=\rho_h*\chi$, with $\phi_{\epsilon,h}$ and letting $h\downarrow0$ therefore gives
\[ \frac12\norm{\phi_\epsilon(s)}_{L^2}^2+\epsilon\int_s^t\norm{\Lambda^{\frac{1+a}{2}}\phi_\epsilon(\tau)}_{L^2}^2\,d\tau=\frac12\norm{\phi_\epsilon(t)}_{L^2}^2+\int_s^t\int_{\R^2}\chi\phi_\epsilon\,dx\,d\tau,\qquad0\leq s\leq t\leq T. \]
In particular, $\phi_\epsilon\in C([0,T];L^2)$. Since $p_c'=2/(1-a)>2$ and $\phi_\epsilon$ is bounded, the interpolation $\norm{f}_{L^{p_c',1}}\leq C_a\norm{f}_{L^2}^{2/p_c'}\norm{f}_{L^\infty}^{1-2/p_c'}$ proves \eqref{eq:renorm-backward-trace}.

\medskip\noindent\emph{Step 3. Duality and the inviscid limit.}\par\nopagebreak
Write $\omega_\epsilon=\curl u_\epsilon$. Integrating $\frac{d}{dt}\pair{\omega_\epsilon}{\phi_\epsilon}=-\pair{\omega_\epsilon}{\chi}$ on a time interval $[\delta, T]$ gives
\[\int_\delta^T\int_{\R^2}\chi\omega_\epsilon\,dx\,dt=\pair{\omega_\epsilon(\delta)}{\phi_\epsilon(\delta)},\qquad 0<\delta<T. \]
Letting $\delta\downarrow0$, the weak-star initial trace for $\omega_\epsilon$ from \cref{prop:eps-family} and the strong initial trace \eqref{eq:renorm-backward-trace} for $\phi_\epsilon$ yield
\begin{equation} \label{eq:renorm-viscous-duality}
\int_0^T\int_{\R^2}\chi\omega_\epsilon\,dx\,dt=\int_{\R^2}\omega_0(x)\phi_\epsilon(0,x)\,dx.
\end{equation}
The uniform bounds and weak equicontinuity from Step~2 yield, after extraction of a further subsequence,
\[ \phi_\epsilon\longrightarrow\phi\quad\text{in }C_{\mathrm{weak}}([0,T];L^p)\quad(1<p<\infty),\qquad\phi_\epsilon\weakstar\phi\quad\text{in }L^\infty_{t,x}. \]
The first convergence means uniform convergence on $[0,T]$ of the pairings with every fixed function in $L^{p'}$. By \cref{prop:compactness}, self-similar scaling, and the uniform bounds, $u_\epsilon\to u$ in $L^1_tL^2_{\rm loc}$ and $\omega_\epsilon\weakstar\omega$ in $L^\infty_tL^{p_c,\infty}_x$. Thus, for $\psi\in C_c^\infty((0,T)\times\R^2)$, as $\epsilon\downarrow0$,
\[ \int_0^T\pair{u_\epsilon\phi_\epsilon-u\phi}{\nabla\psi}\,dt=\int_0^T\pair{u_\epsilon-u}{\phi_\epsilon \nabla\psi}\,dt+\int_0^T\pair{\phi_\epsilon-\phi}{u\cdot\nabla\psi}\,dt\longrightarrow0, \]
\[ \epsilon\abs{\int_0^T\pair{\phi_\epsilon}{\Lambda^{1+a}\psi}\,dt}\leq\epsilon C_\chi\norm{\Lambda^{1+a}\psi}_{L^1_{t,x}}\longrightarrow0. \]
Thus $\phi$ is the unique bounded solution from Step~1 of
\[ -\partial_t\phi-u\cdot\nabla\phi=\chi,\qquad\phi(T)=0,\qquad\phi\in L^\infty_t(L^1_x\cap L^\infty_x). \]
By homogeneity, $\omega_0\1_{B_1}\in L^r$ and $\omega_0\1_{\R^2\setminus B_1}\in L^2$. Hence weak convergence at $t=0$ in $L^{r'}$ and $L^2$ gives
\[ \pair{\omega_0}{\phi_\epsilon(0)-\phi(0)}=\pair{\omega_0\1_{B_1}}{\phi_\epsilon(0)-\phi(0)}+\pair{\omega_0\1_{\R^2\setminus B_1}}{\phi_\epsilon(0)-\phi(0)}\longrightarrow0, \quad\text{as}~ \epsilon \downarrow 0. \]
Passing to the limit $\epsilon\downarrow0$ in \eqref{eq:renorm-viscous-duality}, using also the weak-star convergence of $\omega_\epsilon$, yields
\begin{equation} \label{eq:renorm-limit-duality}
\int_0^T\int_{\R^2}\chi\omega\,dx\,dt=\int_{\R^2}\omega_0(x)\phi(0,x)\,dx.
\end{equation}
Let $w$ be the renormalized solution transported from $\omega_0$ by the flow of $u$. The DiPerna--Lions duality, applied separately to the $L^r$ and $L^2$ parts of $\omega_0$, gives
\[ \int_0^T\int_{\R^2}\chi w\,dx\,dt=\int_{\R^2}\omega_0(x)\phi(0,x)\,dx. \]
Comparison with \eqref{eq:renorm-limit-duality} for arbitrary $\chi$ yields $\omega=w$ almost everywhere. From the renormalization properties, $\omega \in C([0,T];L^r_{\rm loc})$ and the distribution of $\omega(t)$ is conserved for $t\geq 0$.
\end{proof}

\begin{proof}[Proof of \cref{thm:main}]
\Cref{prop:profile-limit} proves part~(a). The associated self-similar field solves the Euler equations on every time interval bounded away from zero. The scale-invariant Lorentz bounds imply that $u$ and $u\otimes u$ are locally integrable up to $t=0$. Using the initial trace from \cref{prop:trace}, we let the lower time boundary tend to zero in the weak formulation and obtain \cref{def:weak}(a). The continuity, convergence, and renormalization assertions in part~(b), together with conservation of the critical Lorentz quasi-norm, follow from \cref{prop:trace,prop:renormalization}.
\end{proof}

\subsection*{Acknowledgements}
The author thanks Prof. Alexandru Ionescu for asking him to extend the range of the parameter $a$ from that in the earlier version. The author also thanks Hyunwoo Kwon for pointing out that Appendix B uses a similar Galerkin scheme to Bradshaw and Tsai \cite{BradshawTsai2017}.

\noindent The author is grateful to Prof. Tom Hou and Xiang Qin for communicating their work \cite{HouQinSongZhang2026}, which led to the addition of the remark on renormalization.

\appendix

\section{Lorentz spaces and the Biot--Savart estimates}
\label{app:lorentz}

This appendix records the Lorentz-space facts used in the paper and proves the two estimates \eqref{eq:HLS}--\eqref{eq:CZ}. All functions below are defined on $\R^2$, unless a different measure space is explicitly indicated.

\subsection{Rearrangements and Lorentz--H\"older inequalities}

For a measurable function $f$, let
\[ \mu_f(\lambda) \triangleq \abs{\{x:\abs{f(x)}>\lambda\}}, \qquad \lambda>0, \]
and let
\[ f^*(t) \triangleq \inf\{\lambda>0:\mu_f(\lambda)\leq t\}, \qquad t>0, \]
be its decreasing rearrangement. For $1\leq p<\infty$ and $1\leq s<\infty$, set
\begin{equation} \label{eq:lorentz-norm}
\norm{f}_{L^{p,s}}\triangleq\left(\int_0^\infty\bigl(t^{1/p}f^*(t)\bigr)^s\,\frac{dt}{t}\right)^{1/s}.
\end{equation}
For $s = \infty$, set
\[ \norm{f}_{L^{p,\infty}} \triangleq \sup_{t>0}t^{1/p}f^*(t) = \sup_{\lambda>0} \lambda\mu_f(\lambda)^{1/p}. \]
Equivalent choices of Lorentz norms change only the constants below.

\begin{proposition}[Basic Lorentz estimates]
\label{prop:lorentz-basic}
The following statements hold.
\begin{enumerate}
\item[(i)] If $1<p<\infty$, $f\in L^{p,\infty}$, and $g\in L^{p',1}$, then
\begin{equation} \label{eq:lorentz-holder}
\int_{\R^2}\abs{f(x)g(x)}\,dx\leq\norm{f}_{L^{p,\infty}}\norm{g}_{L^{p',1}}.
\end{equation}
\item[(ii)] If $1<p_1,p_2<\infty$ and $1\leq r<\infty$ satisfy
\[ \frac1r=\frac1{p_1}+\frac1{p_2}, \]
then
\begin{equation} \label{eq:weak-product}
\norm{fg}_{L^{r,\infty}}\leq C(p_1,p_2)\norm{f}_{L^{p_1,\infty}}\norm{g}_{L^{p_2,\infty}}.
\end{equation}
\item[(iii)] If $E$ has finite measure and $1\leq r<p<\infty$, then
\begin{equation} \label{eq:weak-local-embedding}
\norm{f}_{L^r(E)}\leq C(p,r)\abs{E}^{\frac1r-\frac1p}\norm{f}_{L^{p,\infty}}.
\end{equation}
If $1\leq s<\rho\leq\infty$, then
\begin{equation} \label{eq:strong-to-lorentz}
\norm{h}_{L^{s,1}(E)}\leq C(s,\rho)\abs{E}^{\frac1s-\frac1\rho}\norm{h}_{L^\rho(E)}.
\end{equation}
\item[(iv)] For $1<p<\infty$, the space $L^{p,\infty}$ is the dual of $L^{p',1}$, with equivalent norms, under the integral pairing. Moreover, $C_c^\infty(\R^2)$ is dense in $L^{p',1}$.
\end{enumerate}
\end{proposition}

\begin{proof}
We first prove the rearrangement inequality
\begin{equation} \label{eq:hardy-littlewood-rearrangement}
\int_{\R^2}\abs{fg}\leq\int_0^\infty f^*(t)g^*(t)\,dt.
\end{equation}
Indeed, the layer-cake formula and Tonelli's theorem give
\[ \int_{\R^2}\abs{fg} = \int_0^\infty\int_0^\infty \abs{\{\abs f>\lambda\}\cap\{\abs g>\eta\}} \,d\lambda\,d\eta. \]
The measure of the intersection is at most $\min\{\mu_f(\lambda),\mu_g(\eta)\}$. Applying the layer-cake formula once more on $(0,\infty)$ gives \eqref{eq:hardy-littlewood-rearrangement}. Since
\[ f^*(t) \leq t^{-1/p}\norm{f}_{L^{p,\infty}}, \]
we obtain
\[ \int_{\R^2}\abs{fg} \leq \norm{f}_{L^{p,\infty}} \int_0^\infty t^{-1/p}g^*(t)\,dt = \norm{f}_{L^{p,\infty}}\norm{g}_{L^{p',1}}, \]
because $1/p'-1=-1/p$. This proves (i).

For (ii), let $F=\norm{f}_{L^{p_1,\infty}}$ and $G=\norm{g}_{L^{p_2,\infty}}$. For every $\alpha,\lambda>0$,
\[ \{\abs{fg}>\lambda\} \subset \{\abs f>\alpha\} \cup \{\abs g>\lambda/\alpha\}. \]
Consequently,
\[ \mu_{fg}(\lambda) \leq F^{p_1}\alpha^{-p_1} + G^{p_2}\alpha^{p_2}\lambda^{-p_2}. \]
If $F=0$ or $G=0$, the conclusion is immediate. Otherwise, choosing
\[ \alpha = F^{\frac{p_1}{p_1+p_2}} G^{-\frac{p_2}{p_1+p_2}} \lambda^{\frac{p_2}{p_1+p_2}} \]
shows that
\[ \mu_{fg}(\lambda) \leq 2\left(\frac{FG}{\lambda}\right)^r, \qquad r=\frac{p_1p_2}{p_1+p_2}, \]
which proves \eqref{eq:weak-product}.

To prove the first estimate in (iii), write $M=\norm{f}_{L^{p,\infty}}$ and use
\[ \abs{\{x\in E:\abs{f(x)}>\lambda\}} \leq \min\{\abs E,M^p\lambda^{-p}\}. \]
The layer-cake representation of the $L^r$ norm, split at $\lambda_0=M\abs E^{-1/p}$, gives
\[ \int_E\abs f^r \leq C(p,r)M^r\abs E^{1-r/p}. \]
For the second estimate, extend $h$ by zero outside $E$. Then $h^*(t)=0$ for $t>\abs E$, and H\"older's inequality gives $h^*(t)\leq t^{-1/\rho}\norm{h}_{L^\rho(E)}$. Thus
\[ \norm{h}_{L^{s,1}} \leq \norm{h}_{L^\rho(E)} \int_0^{\abs E}t^{1/s-1/\rho}\,\frac{dt}{t}, \]
which is \eqref{eq:strong-to-lorentz}.

Part (i) shows that every $f\in L^{p,\infty}$ defines a bounded functional on $L^{p',1}$. Conversely, let $\ell$ be a bounded functional on $L^{p',1}$. Its restriction to bounded functions supported on sets of finite measure is represented by integration against a measurable function $f$. Testing with $\operatorname{sgn}(f)\1_E$ gives
\[ \int_E\abs f \leq C\norm{\ell}\abs E^{1/p'}. \]
Taking the supremum over sets $E$ of measure $t$ yields
\[ f^{**}(t) \triangleq \frac1t\int_0^t f^*(s)\,ds \leq C\norm{\ell}t^{-1/p}, \]
and therefore $f\in L^{p,\infty}$. This proves the duality assertion. Finally, truncating a function in value and in space, followed by approximation by simple functions and then by smooth functions, proves density in $L^{p',1}$; the convergence follows directly from \eqref{eq:lorentz-norm}, since its second Lorentz index is finite.
\end{proof}

\begin{lemma}[Test functions]
\label{lem:lorentz-test-functions}
If $1<r<\infty$, $0<\sigma<2$, and $\phi\in C_c^\infty(\R^2)$, then
\[ \Lambda^\sigma\phi\in L^1(\R^2)\cap L^\infty(\R^2) \subset L^{r,1}(\R^2). \]
\end{lemma}

\begin{proof}
For the fractional derivative, use the singular-integral formula
\[ \Lambda^\sigma\phi(x) = c_\sigma\int_{\R^2} \frac{\phi(x)-\phi(y) -\1_{\{\abs{x-y}\leq1\}}\nabla\phi(x)\cdot(x-y)} {\abs{x-y}^{2+\sigma}}\,dy. \]
Taylor's theorem controls the numerator by $C\abs{x-y}^2$ near $y=x$, so the integral is locally bounded because $\sigma<2$. Away from a fixed ball containing $\supp\phi$, one has
\[ \abs{\Lambda^\sigma\phi(x)} \leq C_\phi(1+\abs x)^{-2-\sigma}. \]
Thus $\Lambda^\sigma\phi$ is bounded and integrable.
\end{proof}

\subsection{The estimates for the Biot--Savart law}

Let
\[ \mathcal M f(x) \triangleq \sup_{R>0}\frac1{\abs{B_R}}\int_{B_R(x)}\abs{f(y)}\,dy \]
be the Hardy--Littlewood maximal function. We shall use the standard weak $(1,1)$ estimate
\begin{equation} \label{eq:maximal-weak-one}
\abs{\{\mathcal M f>\lambda\}}\leq\frac{C}{\lambda}\norm{f}_{L^1}.
\end{equation}
For completeness, this follows by choosing, for each point of the superlevel set, a ball on which the average exceeds $\lambda$, applying the Vitali covering lemma to select disjoint balls, and observing that the fivefold dilates cover the superlevel set.

\begin{lemma}\label{lem:maximal-weak-lorentz}
If $1<p<\infty$, then
\[ \norm{\mathcal M f}_{L^{p,\infty}} \leq C_p\norm{f}_{L^{p,\infty}}. \]
\end{lemma}

\begin{proof}
Fix $\lambda>0$ and decompose
\[ f=f_1+f_2,\qquad f_1=f\1_{\{\abs f>\lambda/2\}}. \]
Since $\mathcal M f_2\leq\lambda/2$, the weak $(1,1)$ estimate gives
\[ \abs{\{\mathcal M f>\lambda\}} \leq C\lambda^{-1}\norm{f_1}_{L^1}. \]
If $M=\norm{f}_{L^{p,\infty}}$, then the layer-cake formula yields
\[ \norm{f_1}_{L^1} \leq C_pM^p\lambda^{1-p}. \]
Therefore
$\abs{\{\mathcal M f>\lambda\}}\leq C_pM^p\lambda^{-p}$.
\end{proof}

\begin{proposition}[Biot--Savart estimates in Lorentz spaces]
\label{prop:biot-lorentz}
Let
\[ K(x)=\frac1{2\pi}\frac{x^\perp}{\abs x^2}. \]
If $1<p<2$ and $q=2p/(2-p)$, then
\begin{equation} \label{eq:appendix-HLS}
\norm{K*f}_{L^{q,\infty}}\leq C_p\norm{f}_{L^{p,\infty}}.
\end{equation}
If $1<r<2$ and $s=2r/(2-r)$, then the strong Lorentz refinement also holds:
\[ \norm{K*g}_{L^{s,1}}\leq C_r\norm{g}_{L^{r,1}}. \]
The same bound holds for the adjoint convolution operator.
If $1<p<\infty$, then
\begin{equation} \label{eq:appendix-CZ}
\norm{\nabla(K*f)}_{L^{p,\infty}}\leq C_p\norm{f}_{L^{p,\infty}}.
\end{equation}
The first estimate applied with $p=p_c$ gives \eqref{eq:HLS}, while the second applied with $p=p_c$ gives \eqref{eq:CZ}. The second estimate also justifies the Riesz-transform bound used for the pressure, with $p=1/a$. The strong refinement with $r=q_c'$ and $s=p_c'$ is the adjoint Hardy--Littlewood--Sobolev estimate used in the profile and compactness arguments.
\end{proposition}

\begin{proof}
We first prove \eqref{eq:appendix-HLS}. Since $\abs{K(x)}\leq C\abs{x}^{-1}$, it suffices to estimate
\[ I_1f(x) \triangleq \int_{\R^2}\frac{\abs{f(x-z)}}{\abs z}\,dz. \]
Fix $R>0$. Decomposition into dyadic annuli gives
\[\int_{\abs z\leq R}\frac{\abs{f(x-z)}}{\abs z}\,dz \leq  \sum_{j=0}^\infty \frac{2^{j+1}}{R} \int_{\abs z\leq2^{-j}R}\abs{f(x-z)}\,dz \leq CR\mathcal M f(x).\]
For the far field, put
$h_R(z)=\abs z^{-1}\1_{\{\abs z>R\}}$. Since $p'>2$,
\[ h_R^*(t) \leq C\min\{R^{-1},t^{-1/2}\} \]
and hence
\[ \norm{h_R}_{L^{p',1}} \leq C R^{-1}\int_0^{R^2}t^{1/p'}\,\frac{dt}{t} +C\int_{R^2}^\infty t^{1/p'-1/2}\,\frac{dt}{t} \leq C_pR^{1-2/p}.\]
The Lorentz--H\"older inequality \eqref{eq:lorentz-holder} therefore gives
\[ \int_{\abs z>R}\frac{\abs{f(x-z)}}{\abs z}\,dz \leq C_p\norm{f}_{L^{p,\infty}}R^{1-2/p}. \]
Writing $M=\norm{f}_{L^{p,\infty}}$ and optimizing in $R$ yields
\begin{equation} \label{eq:hedberg}
I_1f(x)\leq C_pM^{1-p/q}\bigl(\mathcal M f(x)\bigr)^{p/q}.
\end{equation}
Indeed, when $\mathcal M f(x)>0$, choose $R=(M/\mathcal M f(x))^{p/2}$; the remaining cases follow by a limit. By \cref{lem:maximal-weak-lorentz},
\[ \abs{\{I_1f>\lambda\}} \leq C_p\left(\frac{M}{\lambda}\right)^q, \]
which proves \eqref{eq:appendix-HLS}.

To prove the strong Lorentz refinement, fix $1<r<2$ and choose $1<r_0<r<r_1<2$. Put $s_i=2r_i/(2-r_i)$ and choose $\theta\in(0,1)$ so that $1/r=(1-\theta)/r_0+\theta/r_1$; then $1/s=(1-\theta)/s_0+\theta/s_1$, where $s=2r/(2-r)$. Applying \eqref{eq:appendix-HLS} at $r_0$ and $r_1$ and using the real-interpolation identities
\[ (L^{r_0,\infty},L^{r_1,\infty})_{\theta,1}=L^{r,1},\qquad (L^{s_0,\infty},L^{s_1,\infty})_{\theta,1}=L^{s,1} \]
gives the claimed estimate.

We next prove \eqref{eq:appendix-CZ}. Each component of $\nabla K*f$ is a second-order Riesz transform, up to a harmless constant multiple of $f$. Denote one such component operator by $T$. Its Fourier multiplier is bounded, so the operator is bounded on $L^2$. Its principal-value kernel $\mathcal K$ satisfies
\[ \abs{\mathcal K(x)}\leq C\abs x^{-2}, \qquad \abs{\nabla\mathcal K(x)}\leq C\abs x^{-3}, \qquad \int_{\Sph^1}\mathcal K(\theta)\,d\theta=0. \]
We recall the Calder\'on--Zygmund argument in the form needed here.
For $h\in L^1\cap L^2$, perform the Calder\'on--Zygmund decomposition
at height $\lambda$:
\[ h=g+\sum_jb_j, \]
where
\[ \norm{g}_{L^\infty}\leq C\lambda,\qquad \norm{g}_{L^1}\leq\norm{h}_{L^1},\qquad \int b_j=0,\qquad \sum_j\norm{b_j}_{L^1}\leq C\norm{h}_{L^1}, \]
and the $b_j$ are supported on disjoint cubes $Q_j$ satisfying
\[ \sum_j\abs{Q_j}\leq C\lambda^{-1}\norm{h}_{L^1}. \]
The $L^2$ bound gives
\[ \abs{\{\abs{Tg}>\lambda/2\}} \leq C\lambda^{-2}\norm{g}_{L^2}^2 \leq C\lambda^{-1}\norm{h}_{L^1}. \]
Outside a fixed dilation $Q_j^*$ of $Q_j$, cancellation and the gradient bound on the kernel give
\[ \int_{\R^2\setminus Q_j^*}\abs{Tb_j(x)}\,dx \leq C\norm{b_j}_{L^1}. \]
Summing in $j$, using Chebyshev's inequality, and adding the measure of $\bigcup_jQ_j^*$ proves the weak $(1,1)$ estimate
\begin{equation} \label{eq:CZ-weak-one}
\abs{\{\abs{Th}>\lambda\}}\leq C\lambda^{-1}\norm{h}_{L^1}.
\end{equation}

These two endpoint estimates imply strong $L^s$ boundedness without any Lorentz-space input. If $1<s<2$, decompose $h=h\1_{\{\abs h>\lambda\}}+h\1_{\{\abs h\leq\lambda\}}$, apply \eqref{eq:CZ-weak-one} to the first term and the $L^2$ estimate to the second. This gives
\[ \abs{\{\abs{Th}>2\lambda\}} \leq C\lambda^{-1} \int_{\{\abs h>\lambda\}}\abs h + C\lambda^{-2} \int_{\{\abs h\leq\lambda\}}\abs h^2. \]
Using this inequality in the distribution formula
\[ \norm{Th}_{L^s}^s = s\int_0^\infty \lambda^{s-1}\abs{\{\abs{Th}>\lambda\}}\,d\lambda, \]
Tonelli's theorem gives
\[ \norm{Th}_{L^s}^s \leq C_s\norm{h}_{L^s}^s. \]
Boundedness for $2<s<\infty$ follows by applying the result to the adjoint and using duality.

Finally, let $f\in L^{p,\infty}$ and choose
$1<r<p<s<\infty$. For $A>0$, write
\[ f=f^{>A}+f^{\leq A}. \]
The distribution-function bound for $f$ gives
\[ \norm{f^{>A}}_{L^r}^r \leq C_{p,r}M^pA^{r-p}, \qquad \norm{f^{\leq A}}_{L^s}^s \leq C_{p,s}M^pA^{s-p}, \qquad M=\norm{f}_{L^{p,\infty}}. \]
Strong $L^r$ and $L^s$ boundedness and Chebyshev's inequality imply
\[ \abs{\{\abs{Tf}>\lambda\}} \leq C M^p \left(\lambda^{-r}A^{r-p} +\lambda^{-s}A^{s-p}\right). \]
Choosing $A=\lambda$ gives
\[ \abs{\{\abs{Tf}>\lambda\}} \leq C_pM^p\lambda^{-p}, \]
which proves \eqref{eq:appendix-CZ}.
\end{proof}

\section{Galerkin construction of hypodissipative profiles}\label{app:galerkin}

This appendix gives a velocity-based Galerkin construction of a hypodissipative self-similar profile.

\begin{theorem} \label{thm:galerkin-profile}
Let $0<a<1$ and let $v_0$ be a locally Lipschitz, divergence-free, $(-a)$-homogeneous vector field on $\R^2\setminus\{0\}$. Then there are a divergence-free vector field $W$ and a distribution $Q$ such that
\begin{equation} \label{eq:galerkin-profile}
\Lambda^{1+a}W-\frac1{1+a}\bigl(aW+y\cdot\nabla W\bigr)+W\cdot\nabla W+\nabla Q=0
\end{equation}
in distributions on $\R^2$ and
\[ W-e^{-\Lambda^{1+a}}v_0\in H^{\frac{1+a}{2}}(\R^2). \]
Moreover, the associated self-similar velocity and pressure are defined by
\[ v(t,x)=t^{-\frac{a}{1+a}}W\left(t^{-\frac1{1+a}}x\right),\qquad q(t,x)=t^{-\frac{2a}{1+a}}Q\left(t^{-\frac1{1+a}}x\right). \]
Then $v$ is a weak solution of
\[ \partial_t v+v\cdot\nabla v+\nabla q+\Lambda^{1+a}v=0,\qquad \diver v=0, \]
with initial datum $v_0$, and $v(t)\to v_0$ strongly in $L^2_{\rm loc}(\R^2)$ as $t\downarrow0$. Here weak solution means that, for every divergence-free $\varphi\in C_c^\infty([0,\infty)\times\R^2;\R^2)$,
\[ \int_0^\infty\!\int_{\R^2}\left(v\cdot\partial_t\varphi+(v\otimes v):\nabla\varphi-v\cdot\Lambda^{1+a}\varphi\right)\,dx\,dt+\int_{\R^2}v_0(x)\cdot\varphi(0,x)\,dx=0. \]
\end{theorem}

\begin{proof}
Let $\Pcal$ denote the Leray projection and set
\[ \mathcal L\triangleq\Lambda^{1+a}-\frac{a}{1+a}-\frac1{1+a}y\cdot\nabla,\qquad V_0\triangleq e^{-\Lambda^{1+a}}v_0. \]
The projected profile equation is
\[ \mathcal LW+\Pcal\diver(W\otimes W)=0. \]
\smallskip\noindent\emph{Step 1. The linear profile.}\par
The homogeneity of $v_0$ and the fractional heat-kernel bounds give
\begin{equation} \label{eq:galerkin-V0-bounds}
\abs{V_0(y)}\leq C(1+\abs y)^{-a},\qquad \abs{\nabla^kV_0(y)}\leq C_k(1+\abs y)^{-a-1}\quad(k\geq1).
\end{equation}
For completeness, put $s=1+a$ and let $p_t$ be the kernel of $e^{-t\Lambda^s}$. The standard fractional heat-kernel bounds give
\[ \abs{\nabla^jp_t(x)}\leq C_jt\bigl(t^{1/s}+\abs x\bigr)^{-2-s-j},\qquad \norm{\nabla^jp_t}_{L^1}\leq C_jt^{-j/s}. \]
Choose a radial $\chi\in C_c^\infty(\{1/4<\abs x<4\})$ that equals one near $\Sph^1$, and write $v_0=f+g$ with $f=\chi v_0$. Since $f\in W_c^{1,\infty}$ and $g$ is supported a positive distance from $\Sph^1$, for $0<t\leq1$ and $\theta\in\Sph^1$ one has
\[ \abs{p_t*f(\theta)}\leq C,\qquad \abs{\nabla^k(p_t*f)(\theta)}\leq\norm{\nabla^{k-1}p_t}_{L^1}\norm{\nabla f}_{L^\infty}\leq C_kt^{-(k-1)/s}\quad(k\geq1). \]
On the other hand, the pointwise kernel bound and $\abs{g(y)}\leq C\abs y^{-a}$ give
\[ \abs{\nabla^k(p_t*g)(\theta)}\leq C_kt\int_{\supp g}\abs{\theta-y}^{-2-s-k}\abs y^{-a}\,dy\leq C_k\quad(k\geq0). \]
The integral is finite both near the origin, because $a<2$, and at infinity, while the separation of $\supp g$ from $\Sph^1$ removes the kernel singularity. Consequently,
\[ \abs{e^{-t\Lambda^s}v_0(\theta)}\leq C,\qquad \abs{\nabla^ke^{-t\Lambda^s}v_0(\theta)}\leq C_kt^{-(k-1)/s}\quad(k\geq1). \]
By homogeneity,
\[ V_0(Rx)=R^{-a}\bigl(e^{-R^{-s}\Lambda^s}v_0\bigr)(x), \]
so for $R\geq1$ and $\theta\in\Sph^1$,
\[ \abs{V_0(R\theta)}\leq CR^{-a},\qquad \abs{\nabla^kV_0(R\theta)}\leq C_kR^{-a-k}R^{k-1}=C_kR^{-a-1}. \]
On bounded sets all derivatives of $V_0=p_1*v_0$ are bounded, which proves \eqref{eq:galerkin-V0-bounds}. Differentiating
\[ e^{-t\Lambda^{1+a}}v_0(x)=t^{-\frac{a}{1+a}}V_0\left(t^{-\frac1{1+a}}x\right) \]
at $t=1$ yields $\mathcal LV_0=0$.

\smallskip\noindent\emph{Step 2. A small far-field background.}\par
Since $V_0$ is smooth and divergence free, write $V_0=\nabla^\perp\psi_0$ with $\psi_0(0)=0$. Estimate \eqref{eq:galerkin-V0-bounds} gives $\abs{\psi_0(y)}\leq C(1+\abs y)^{1-a}$. Choose a radial cutoff $\chi_R$ that vanishes on $B_R$ and equals one outside $B_{2R}$, and define
\[ V_1\triangleq\nabla^\perp(\chi_R\psi_0),\qquad V_2\triangleq V_1-V_0. \]
Then $V_1$ is divergence free, $V_2\in C_c^\infty(\R^2)$, and the product rule together with \eqref{eq:galerkin-V0-bounds} gives
\begin{equation} \label{eq:galerkin-small-background}
\norm{\nabla V_1}_{L^\infty}\leq CR^{-1-a}.
\end{equation}
Seek $W=V_1+Z$. Since $\mathcal LV_0=0$ and $V_2$ is smooth and compactly supported,
\[ F\triangleq-\mathcal LV_1-\Pcal\diver(V_1\otimes V_1)=-\mathcal LV_2-\Pcal\diver(V_1\otimes V_1)\in L^2(\R^2). \]
Here $\Lambda^{1+a}V_2(y)=O(\abs y^{-3-a})$, while $V_1\cdot\nabla V_1=O(\abs y^{-1-2a})$; both are square integrable because $a>0$. The equation for $Z$ is
\begin{equation} \label{eq:galerkin-Z-equation}
\mathcal LZ+\Pcal\diver\bigl(Z\otimes Z+V_1\otimes Z+Z\otimes V_1\bigr)=F.
\end{equation}
\smallskip\noindent\emph{Step 3. Galerkin approximation.}\par
Choose a countable family $\{\phi_j\}_{j\geq1}\subset C_c^\infty(\R^2;\R^2)$ of divergence-free fields whose span is dense in the divergence-free subspace of $H^{\frac{1+a}{2}}$ and has the following local approximation property: every divergence-free $\phi\in C_c^\infty(\R^2;\R^2)$ is the $C^\infty$ limit of a sequence in this span whose supports lie in one fixed compact set. Such a family is obtained by taking, for every positive rational radius, a countable $C^\infty$-dense family of compactly supported scalar functions in the corresponding ball, applying $\nabla^\perp$, and enumerating the union. Every compactly supported divergence-free test field has a compactly supported smooth stream function, so the stated approximation property follows.

Let $X_N=\operatorname{span}\{\phi_1,\ldots,\phi_N\}$, let $P_N$ be the $L^2$-orthogonal projection onto $X_N$, and define the continuous Galerkin map $\mathcal G_N:X_N\to X_N$ by
\[ \mathcal G_N(Z)=P_N\mathcal LZ+P_N\Pcal\diver\bigl(Z\otimes Z+V_1\otimes Z+Z\otimes V_1\bigr)-P_NF. \]
For $Z\in X_N$, integration by parts gives
\[ \pair{\mathcal LZ}{Z}=\norm{\Lambda^{\frac{1+a}{2}}Z}_{L^2}^2+\frac{1-a}{1+a}\norm{Z}_{L^2}^2. \]
The divergence-free identities give
\[ \int_{\R^2}\bigl((Z\cdot\nabla)Z+(V_1\cdot\nabla)Z\bigr)\cdot Z\,dy=0,\qquad \abs{\int_{\R^2}(Z\cdot\nabla V_1)\cdot Z\,dy}\leq\norm{\nabla V_1}_{L^\infty}\norm{Z}_{L^2}^2. \]
Choose $R$ in \eqref{eq:galerkin-small-background} so large that $\norm{\nabla V_1}_{L^\infty}<(1-a)/(2(1+a))$. It follows that
\[ \pair{\mathcal G_N(Z)}{Z}\geq\norm{\Lambda^{\frac{1+a}{2}}Z}_{L^2}^2+\frac{1-a}{2(1+a)}\norm{Z}_{L^2}^2-\norm F_{L^2}\norm Z_{L^2}. \]
Choose $\rho>2(1+a)\norm F_{L^2}/(1-a)$. Consequently, $\pair{\mathcal G_N(Z)}{Z}>0$ on the $L^2$-sphere $\{Z\in X_N:\norm Z_{L^2}=\rho\}$. Brouwer's lemma gives a zero $Z_N$ of $\mathcal G_N$ inside this sphere. Evaluating the preceding inequality at $Z_N$ yields
\begin{equation} \label{eq:galerkin-uniform-bound}
\sup_N\norm{Z_N}_{H^{\frac{1+a}{2}}}<\infty.
\end{equation}
\smallskip\noindent\emph{Step 4. Passage to the limit.}\par
After passing to a subsequence, $Z_N\rightharpoonup Z$ in $H^{\frac{1+a}{2}}(\R^2)$. Since $\frac{1+a}{2}>\frac{1}{2}$, Rellich compactness and the local Sobolev embedding give $Z_N\to Z$ strongly in $L^4_{\rm loc}(\R^2)$. For each fixed $j$, the Galerkin identity can therefore be passed to the limit against $\phi_j$: the strong local convergence handles the quadratic terms, the linear fractional term passes by weak convergence in $H^{\frac{1+a}{2}}$, and the dilation term is handled by moving its derivative onto $\phi_j$. The identity holds on the span of the $\phi_j$. The fixed-support $C^\infty$ approximation property then extends it to every divergence-free $\phi\in C_c^\infty(\R^2;\R^2)$. Thus \eqref{eq:galerkin-Z-equation} holds in distributions against divergence-free tests. The complementary Helmholtz projection recovers a distributional pressure $Q$, and $W=V_1+Z$ solves \eqref{eq:galerkin-profile}. Finally,
\[ W-V_0=V_2+Z\in H^{\frac{1+a}{2}}(\R^2). \]
\smallskip\noindent\emph{Step 5. The initial trace.}\par
Set $\mathsf{U}=W-V_0$. Scaling gives
\[ v(t,x)=e^{-t\Lambda^{1+a}}v_0(x)+t^{-\frac{a}{1+a}}\mathsf{U}\left(t^{-\frac1{1+a}}x\right) \]
and
\[ \norm{t^{-\frac{a}{1+a}}\mathsf{U}(t^{-\frac1{1+a}}\,\cdot\,)}_{L^2}=t^{\frac{1-a}{1+a}}\norm{\mathsf{U}}_{L^2}\longrightarrow0\qquad\text{as }t\downarrow0. \]
To treat the linear term on a compact set $K$, choose $\chi\in C_c^\infty$ equal to one near $K\cup\{0\}$. Since $\chi v_0\in L^2(\R^2)$, strong continuity of the fractional heat semigroup gives $e^{-t\Lambda^{1+a}}(\chi v_0)\to\chi v_0$ in $L^2$. The heat-kernel tail, the separation of $K$ from $\supp(1-\chi)$, and boundedness of $v_0$ on $\supp(1-\chi)$ give
\[ \sup_{x\in K}\abs{e^{-t\Lambda^{1+a}}\bigl((1-\chi)v_0\bigr)(x)}\leq C_Kt,\qquad 0<t\leq1. \]
Hence $v(t)\to v_0$ in $L^2(K)$. The same estimates give $v\in L^2_{\rm loc}([0,\infty)\times\R^2)$ and $v\otimes v\in L^1_{\rm loc}$. It remains only to justify the nonlocal dissipative term at $t=0$. Homogeneity gives $v_0\in L^{2/a,\infty}$, and the fractional heat semigroup is uniformly bounded on this Lorentz space. Lorentz duality and the global $L^2$ estimate for the correction therefore give, for every compactly supported smooth spacetime test field $\varphi$,
\[ \abs{\pair{e^{-t\Lambda^{1+a}}v_0}{\Lambda^{1+a}\varphi(t)}}\leq C_\varphi\norm{v_0}_{L^{2/a,\infty}},\qquad \abs{\pair{t^{-\frac{a}{1+a}}\mathsf{U}(t^{-\frac1{1+a}}\,\cdot\,)}{\Lambda^{1+a}\varphi(t)}}\leq C_\varphi t^{\frac{1-a}{1+a}}\norm{\mathsf{U}}_{L^2}. \]
Indeed, $\Lambda^{1+a}\varphi(t)\in L^{2/(2-a),1}\cap L^2$ uniformly in $t$. Both bounds are integrable at the origin. On every interval $[\delta,T]$, the self-similar change of variables carries the profile equation to the physical hypodissipative equation. Testing there against a divergence-free spacetime field and sending $\delta\downarrow0$, using the strong local trace for the boundary term and the preceding estimates for the dissipative term, proves the stated weak formulation.
\end{proof}

\begin{remark}
In the range $\frac{1}{3} < a < 1$, the profile $W$ constructed by the Galerkin method is smooth, with regularity established by bootstrapping through Duhamel's formula.
\end{remark}

\end{document}